\documentclass[secnum,leqno,11pt]{article}%%% eqno resets every section 
\makeatletter
\renewcommand{\section}{%
\@startsection{section}%
{1}{\z@}{-3.5ex \@plus -1ex \@minus -.2ex}%
{2.3ex \@plus.2ex}{\normalfont\large\bfseries}%
}\makeatother
\newcommand{\toukou}[1]{\ifx\TOUKOU\undefined\else{#1}\fi}%
\newcommand{\toukoudel}[1]{\ifx\TOUKOU\undefined{#1}\else\fi}%
\newcommand{\toukouchange}[2]{\ifx\TOUKOU\undefined{#1}\else{#2}\fi}%
\usepackage{amssymb}
\usepackage{amsmath}
\usepackage{amsthm}
\usepackage[dvipdfm]{graphicx} 
\usepackage{verbatim,enumerate}
\usepackage[usenames]{color}
 \newtheorem{theorem}{Theorem}[section]
 
 \newtheorem{lemma}[theorem]{Lemma}
 \newtheorem{corollary}[theorem]{Corollary}
\theoremstyle{definition}
 \newtheorem{definition}[theorem]{Definition}
 
 \newtheorem*{acknowledgements}{Acknowledgements}
 
\numberwithin{equation}{section}
\renewcommand{\labelenumi}{\theenumi}
\newcommand{\R}{\boldsymbol{R}}
\newcommand{\C}{\boldsymbol{C}}

\newcommand{\rank}{\operatorname{rank}}

\renewcommand{\phi}{\varphi}
\newcommand{\inner}[2]{\left\langle{#1},{#2}\right\rangle}
\newcommand{\spann}[1]{\left\langle{#1}\right\rangle_{\R}}

\newcommand{\corank}{\operatorname{corank}}
\newcommand{\hess}{\operatorname{Hess}}

\newcommand{\A}{\mathcal{A}}
\newcommand{\E}{\mathcal{E}}
\newcommand{\M}{\mathcal{M}}
\newcommand{\HH}{\mathcal{H}}
\newcommand{\ep}{\varepsilon}
\newcommand{\pmt}[1]{{\begin{pmatrix} #1  \end{pmatrix}}}
\newcommand{\trans}[1]{\vphantom{#1}^t\!#1}

\begin{document}
\toukoudel{
\vspace*{10mm}

\begin{center}
{\large {\bf Criteria for Morin singularities for maps 
into lower dimensions, and applications}}
\\[2mm]
{\today}
\\[2mm]
Kentaro Saji\\
\end{center}
\renewcommand{\thefootnote}{\fnsymbol{footnote}}
\footnote[0]{
2010 Mathematics Subject classification.
Primary 57R45; Secondary 58K60, 58K65.}
\footnote[0]{
Partly supported by 
Japan Society for the Promotion of Science (JSPS)
and 
Coordenadoria de Aperfei\c{c}oamento de Pessoal de N\'ivel Superior
under the Japan-Brazil research cooperative
program and
Grant-in-Aid for Scientific Research (C) No. 26400087,
from JSPS.}
\footnote[0]{Keywords and Phrases. Morin singularities, criteria}
%\tableofcontents      %optional
}
\begin{abstract}      %optional
We give criteria for Morin singularities for germs of maps
into lower dimensions.
As an application, we study the bifurcation 
of Lefschetz singularities.
\end{abstract}
\toukou{
\title[Criteria for Morin singularities]
{Criteria for Morin singularities for maps 
into lower dimensions, and applications}
\author[K. Saji]{Kentaro Saji}
\address{Department of Mathematics, 
Graduate School of Science, 
Kobe University, 
Rokko, Nada, Kobe 657-8501, Japan}
\email{sajiO\!\!\!amath.kobe-u.ac.jp}
\thanks{Partly supported by 
Japan Society for the Promotion of Science (JSPS)
and 
Coordenadoria de Aperfei\c{c}oamento de Pessoal de N\'ivel Superior
under the Japan-Brazil research cooperative
program and
Grant-in-Aid for Scientific Research (C) No. 26400087,
from JSPS}
\dedicatory{Dedicated to Professor Mar\'ia del Carmen Romero Fuster 
on the occasion of her sixtieth birthday}
\subjclass[2010]{Primary 57R45; Secondary 58K60, 58K65}
\date{\today}
\maketitle
\bibliographystyle{amsplain}
}
\section{Introduction}
A map-germ
$f:(\R^m,0)\to(\R^n,0)$ $(m>n)$ is called a
{\em $k$-Morin singularity\/}
($1\leq k\leq n$)
if it is $\A$-equivalent to
the following map-germ at the origin:
\begin{equation}\label{eq:morinnor1}
\begin{array}{l}
h_{0,k}(x_1,\ldots,x_{n-1},y_1,\ldots,y_{m-n},z)\\
\hspace{20mm}
\displaystyle
=\bigg(x_1,\ldots,x_{n-1},q(y_1,\ldots,y_{m-n})
+
z^{k+1}+\sum_{i=1}^{k-1}x_iz^i\bigg)
\end{array}
\end{equation}
if $k\geq2$, and 
\toukouchange{
$
h_{0,1}(x_1,\ldots,x_{n-1},y_1,\ldots,y_{m-n+1})
=\big(x_1,\ldots,x_{n-1},q(y_1,\ldots,y_{m-n+1})\big)
$}{
$
h_{0,1}(x_1,\ldots,x_{n-1},y_1,\ldots,y_{m-n+1})
{=}\big(x_1,\ldots,x_{n-1},q(y_1,\ldots,y_{m-n+1})\big)
$}
if $k=1$,
where $q$ is a non-degenerate quadratic germ of function.
The $1$-Morin singularity is also called the {\em fold},
and
the $2$-Morin singularity is also called the {\em cusp}.
We say that two map-germs 
$f,g:(\R^m,0)\to(\R^n,0)$ are {\em $\A$-equivalent\/}
if there exist germs of
diffeomorphism
$\phi:(\R^m,0)\to(\R^m,0)$  and \
$\Phi:(\R^n,0)\to(\R^n,0)$ such  that
$\Phi\circ f\circ \phi=g$.
Morin singularities are stable, and conversely, all corank one and stable
map-germs are Morin singularities. 
This means that Morin singularities are fundamental and
frequently appear as singularities of
maps from one manifold to another.
If $\corank df_0=1$, then
one can choose a coordinate system $(x,y)$
such that $f(x,y)=
\big(x_1,\ldots,x_{n-1},h(x,y)\big)$,
where 
$x=(x_1,\ldots,x_{n-1})$,
$y=(y_1,\ldots,y_{m-n+1})$.
We call this procedure a {\em normalization}.
Morin \cite{morin} gave a characterization of those singularities
in terms of
transversality of the jet extensions to 
the Thom-Boardman singularity set,
and also gave criteria for germs with respect to
a normalized form
$\big(x_1,\ldots,x_{n-1},h(x,y)\big)$.
Morin singularities are also characterized using
the intrinsic derivative due to 
Porteous (\cite{porteous} see also \cite{ando,GG}).
Criteria for singularities without using
normalization are not only more convenient 
but also indispensable in some cases.
We refer to
criteria which are independent of normalization 
as {\em general criteria}.
In fact, in the case of wave front
surfaces in 3-space, general criteria for 
cuspidal edges and swallowtails were given
in \cite{krsuy},
where we used them to study 
the local and global behavior of flat fronts 
in hyperbolic 3-space.
Recently general criteria for other singularities and 
several of their applications have been given
in \cite{ishimachi,ist,kruy,nishi,jgea,suyak,syy}.
In this paper, we give general criteria for Morin singularities.
Using them, we give applications to bifurcation of the
Lefschetz singularity which plays important roles
in low-dimensional topology.
See \cite{df,fukuda,iikt,saeki,saekisaku,sajihigh,sajiiso,sst,szu} for other 
investigations of Morin singularities.

\section{Singular sets and Hesse matrix of corank one singularities}
\begin{definition}
Let
$f:(\R^m,0)\to(\R^n,0)$
be a map-germ and denote by $S(f)$ the singular locus of $f$.
A collection of vector fields 
\begin{equation}\label{eq:collecvf}
(\xi,\eta)=(\xi_1,\ldots,\xi_{n-1},\eta_1,\ldots,\eta_{m-n+1})
\end{equation}
on $(\R^m,0)$ 
is said to be {\em adapted with respect to}\/ $f$ if
$\xi_1,\ldots,\xi_{n-1},\eta_1,\ldots,\eta_{m-n+1}$ generates
$T_0\R^m$ at $0$, and
$\spann{\eta_1(p),\ldots,\eta_{m-n+1}(p)}=\ker df_p$
for any $p\in S(f)$ near $0$.
\end{definition}
\begin{lemma}\label{adaptedexist}
Let\/ $f:(\R^m,0)\to(\R^n,0)$ be a map-germ satisfying\/
$\rank df_0=n-1$.
Then there exists 
a collection of
vector fields\/ 
$(\xi,\eta)$ as in\/ \eqref{eq:collecvf}
which is adapted with respect to\/ $f$.
\end{lemma}
\begin{proof} 
\toukouchange
{
Since the result does not depend on the choice of coordinate system
and $\rank df_0=n-1$, then
we can take 
a coordinate system
}
{
\hfill Since\hfill  the\hfill  result\hfill  does\hfill  
not\hfill  depend\hfill on\hfill the\hfill choice\hfill 
of\hfill coordinate\\
system\hfill 
and\hfill $\rank df_0=n-1$,\hfill then\hfill
we\hfill can\hfill take\hfill 
a\hfill coordinate\hfill system}
$(x,y)=(x_1,\ldots,x_{n-1},y_1,\ldots,y_{m-n+1})$
in a neighborhood of the origin $U$
on the source space, such that
\begin{equation}\label{eq:normalcoord}
f(x,y)
=
(x,h(x,y)),\quad dh_0=0.
\end{equation}
\toukouchange{
Then \hfill
$S(f)=\{(x,y)\in U\,|\,h_{y_1}(x,y)=\cdots=h_{y_{m-n+1}}(x,y)=0\}$\hfill 
holds. \hfill
Thus\hfill $\partial x_1,\ldots,\partial x_{n-1},$\\
$\partial y_1,\ldots,\partial y_{m-n+1}$
are the desired vector fields.}
{Then $S(f)=\{(x,y)\in U\,|\,h_{y_1}(x,y)=\cdots=h_{y_{m-n+1}}(x,y)=0\}$ holds.
Thus $\partial x_1,\ldots,\partial x_{n-1},$$
\partial y_1,\ldots,\partial y_{m-n+1}$
are the desired vector fields.}
\end{proof}
Let
$f:(\R^m,0)\to(\R^n,0)$
be a map-germ satisfying $\rank df_0=n-1$,
and
$(\xi,\eta)=(\xi_1,\ldots,\xi_{n-1},\eta_1,\ldots,\eta_{m-n+1})$
an adapted collection of vector fields with respect to $f$.
Set
$$
\lambda_i=\det(\xi_1f,\ldots,\xi_{n-1}f,\eta_if),\quad
i=1,\ldots,m-n+1
$$
and
$$
\Lambda=(\lambda_1,\ldots,\lambda_{m-n+1}),$$
where $\zeta f$ stands for the directional derivative of
$f$ along the vector field $\zeta$.
Then 
$S(f)=\{\Lambda=0\}$.
\begin{definition}
Let $0$ be a singular point of $f=(f_1,\ldots,f_n):(\R^m,0)\to(\R^n,0)$
satisfying $\rank df_0=n-1$.
We say that $0$ is {\em non-degenerate\/} if
$
\rank d\Lambda_0=m-n+1
$.
\end{definition}
This condition is a special case of the condition called
{\em critical normalization}.
See \cite{dgw} for details.
\begin{lemma}\label{lem:nondegnotdep}
The non-degeneracy condition above does not depend on the choice of
coordinate systems on the source space nor
on the target space.
\end{lemma}
\begin{proof}
One can easily show that it does not depend on the coordinate
system on the target.
In fact, 
let $\Phi:(\R^n,0)\to(\R^n,0)$ be a germ of diffeomorphism,
and we regard $d\Phi_x$ as the matrix representation of $d\Phi_x$
with respect to the standard basis at $x\in\R^n$.
Set 
$\overline{\lambda}_i=\det\big(\xi_1(\Phi\circ f),\ldots,
\xi_{n-1}(\Phi\circ f),\eta_i(\Phi\circ f)\big)$,
and 
$\overline{\Lambda}=\big(
\overline{\lambda}_1,\ldots,\overline{\lambda}_{m-n+1}\big).$
Then $\overline\Lambda(x)=\det(d\Phi_{f(x)})\Lambda(x)$ holds.
Thus $\rank d\Lambda_0$ does not depend on the choice of the coordinate system
on the target.

Secondly, we show that it does not depend on the choice of 
an adapted collection of vector fields.
Since it does not depend on the coordinate
system on the target,
we may assume that $f=(f_1,\ldots,f_n)$ satisfies
$d(f_n)_0=0$.
Then for any vector field $\zeta$,
it holds that
$\zeta\lambda_i=\det(\xi_1f,\ldots,\xi_{n-1}f,\zeta\eta_i f)(0)
=\big(\Delta\zeta\eta_if_n\big)(0)$,
where $\Delta=\det(\xi_1\hat{f},\ldots,\xi_{n-1}\hat{f})$,
and $\hat{f}=(f_1,\ldots,f_{n-1})$.
Let 
$
(\overline{\xi}_1,\ldots,\overline{\xi}_{n-1},
\overline{\eta}_1,\ldots,\overline{\eta}_{m-n+1})
$
be an adapted collection of vector fields satisfying
\begin{equation}\label{eq:vfchangeeta}
\begin{array}{l}
\pmt{\bar\xi_1\\ \vdots\\ \bar\xi_{n-1}\\
\hline
\bar\eta_1\\ \vdots\\ \bar\eta_{m-n+1}}
=
\left(
\begin{array}{c|c}
A^1&A^2\\
\hline
B^1&B^2
\end{array}
\right)
\pmt{\xi_1\\ \vdots\\ \xi_{n-1}\\
\hline
\eta_1\\ \vdots\\ \eta_{m-n+1}},
\quad\text{where}\quad
\left(
\begin{array}{c|c}
A^1&A^2\\
\hline
B^1&B^2
\end{array}
\right)\\[14mm]
\hfill
=
\left(
\begin{array}{ccc|ccc}
a^1_{1,1}&\cdots&a^1_{1,n-1}&a^2_{1,1}&\cdots&a^2_{1,m-n+1}\\
\vdots&\ddots&\vdots&\vdots&\ddots&\vdots\\
a^1_{n-1,1}&\cdots&a^1_{n-1,n-1}&a^2_{n-1,1}&\cdots&a^2_{n-1,m-n+1}\\
\hline
b^1_{1,1}&\cdots&b^1_{1,n-1}&b^2_{1,1}&\cdots&b^2_{1,m-n+1}\\
\vdots&\ddots&\vdots&\vdots&\ddots&\vdots\\
b^1_{m-n+1,1}&\cdots&b^1_{m-n+1,n-1}&
b^2_{m-n+11}&\cdots&b^2_{m-n+1,m-n+1}
\end{array}
\right),
\end{array}
\end{equation}
where $A^1,B^2$ are regular matrices at $0$, and $B^1=O$ on $S(f)$.
Set 
\begin{equation}\label{eq:lambdabar}
\bar\lambda_i=
\det\big(\bar\xi_1f,\ldots,\bar\xi_{n-1}f,\bar\eta_if\big)\quad
(i=1,\ldots,m-n+1),\quad
\bar\Lambda
=\big(\bar\lambda_1,\ldots,\bar\lambda_{m-n+1}\big).
\end{equation}
Then for any vector field $\zeta$, we see that
$$
\zeta\overline{\lambda}_i(0)
=
\big(\det A^1\,\Delta\, \zeta\bar\eta_i f_n\big)(0),$$
and
$
d\overline\Lambda_0=\big((\det A^1\Delta)B^2d\Lambda\big)(0).
$
Thus we have the conclusion.
\end{proof}
For a non-degenerate singularity $0$,
we define a matrix $\HH_\eta$ by
\begin{equation}\label{eq:hheta}
\HH_\eta=
\pmt{\eta_j\lambda_i}_{1\leq i,j\leq m-n+1}
=
\pmt{
\eta_1\lambda_1&\ldots&\eta_1\lambda_{m-n+1}\\
\vdots&\ddots&\vdots\\
\eta_{m-n+1}\lambda_1&\ldots&\eta_{m-n+1}\lambda_{m-n+1}}.
\end{equation}
Then $\HH_\eta$ is symmetric on $S(f)$.
In fact,
since $[\eta_j,\eta_i](p)\in T_p\R^m$, 
there exist functions $\alpha_i$ $(i=1,\ldots,n-1)$ and
$\beta_j$ $(j=1,\ldots,m-n+1)$ 
such that
\begin{equation}\label{eq:symm1}
[\eta_j,\eta_i](p)
=
\sum_{i=1}^{n-1}\alpha_i\xi_i(p)
+
\sum_{j=1}^{m-n+1}\beta_j\eta_j(p).
\end{equation}
If
$p\in S(f)$, then by $\eta_j f(p)=0$ and \eqref{eq:symm1}
it follows that
$$
\eta_j\lambda_i=
\det(\xi_1f,\ldots,\xi_{n-1}f,\eta_j\eta_if)
=
\det(\xi_1f,\ldots,\xi_{n-1}f,\eta_i\eta_jf)
=
\eta_i\lambda_j
$$
on $S(f)$.
\begin{lemma}\label{lem:rankhf}
Let\/ $0$ be a non-degenerate singular point of\/
$f=(f_1,\ldots,f_n):(\R^m,0)\to(\R^n,0)$.
The matrix-valued function\/ $\HH_\eta$ on\/ $S(f)$
does not depend on the choice of 
an adapted collection of vector fields
with respect to\/ $f$,
nor on the coordinate systems on the target
up to non-zero functional multiplications.
In particular, $\rank \HH_\eta$ on\/ $S(f)$
does not depend on the choice of 
adapted collections of vector fields with respect to\/ $f$
nor on the coordinate systems on the target.
\end{lemma}
\begin{proof}
Let  
$(\xi_1,\ldots,\xi_{n-1},\eta_1,\ldots,\eta_{m-n+1})$ 
and
$(\bar\xi_1,\ldots,\bar\xi_{n-1},
\bar\eta_1,\ldots,\bar\eta_{m-n+1})$
be adapted collections of vector fields with respect to $f$
satisfying \eqref{eq:vfchangeeta}.
By the conditions, it holds that 
$B^1=0$, $\det A^1\ne0$ and $\det B^2\ne0$ on $S(f)$.
Set
$$
\bar\lambda_i=
\det(\bar\xi_1f,\ldots,\bar\xi_{n-1}f,\bar\eta_if),\quad
\HH_{\bar\eta}=
\pmt{\bar\eta_j\bar\lambda_i}_{1\leq i,j\leq m-n+1}.
$$
Since
$\eta_i$ and $\bar\eta_i$ $(i=1,\ldots,m-n+1)$ 
are included in $\ker df$ on $S(f)$,
one can see that:
$$
\begin{array}{rcl}
\displaystyle
\bar\eta_j\bar\lambda_i
&=&
\displaystyle
\bar\eta_j
\det(\bar\xi_1f,\ldots,\bar\xi_{n-1}f,\bar\eta_if)\\
&=&
\displaystyle
\inner{\bar\xi_1f\times\cdots\times\bar\xi_{n-1}f}
{\bar\eta_j\bar\eta_if}\\
&=&
\displaystyle
\det A^1 \inner{\xi_1f\times\cdots\times\xi_{n-1}f}
{\bar\eta_j\bar\eta_if}\\
&=&
\displaystyle
\det A^1 \inner{\xi_1f\times\cdots\times\xi_{n-1}f}
{\sum_{k,l}\eta_l b^1_{i,k}\,\xi_l f
+
\sum_{k,l}b^2_{jl}b^2_{ik}\eta_l\eta_kf}\\
&=&
\displaystyle
\det A^1 
\sum_{k,l}b^2_{jl}b^2_{ik}
\inner{\xi_1f\times\cdots\times\xi_{n-1}f}
{\eta_l\eta_kf}\\
&=&
\displaystyle
\det A^1 
\sum_{k,l}b^2_{jl}b^2_{ik}
\,\eta_l\lambda_k.
\end{array}
$$
Thus on $S(f)$, we have that 
\begin{equation}\label{eq:hfkankei}
\HH_{\bar\eta}
=
(\det A^1)^{m-n+1}(\det B^2) \HH_\eta.
\end{equation}
This proves the first assertion.
One can show the independence 
for the target coordinate systems easily
by following the same method as used in 
the proof of Lemma \ref{lem:nondegnotdep}.
\end{proof}
If $0$ is a non-degenerate singularity, then
$S(f)$ is a manifold.
Thus we can consider $g=f|_{S(f)}$.
Then we have the following lemma.
\begin{lemma}\label{lem:s2}
Let\/ $0$ be a non-degenerate singular point of\/
$f=(f_1,\ldots,f_n)$.
Then
$$
S(f|_S)=S(g)=
\{p\in S(f)\,|\,\det \HH_\eta(p)=0\}$$
near\/ $0$. Moreover, 
by the identification 
\begin{equation}\label{eq:hh}
\HH_\eta:\sum_{i=1}^{m-n+1}a_i\eta_i
\mapsto 
\sum_{i=1}^{m-n+1}\Bigg(\sum_{j=1}^{m-n+1}a_j\eta_i\lambda_j\Bigg)\eta_i,
\end{equation}
it holds that\/
$\ker dg_p=\ker \HH_\eta(p)=\ker df_p\cap T_pS(f)$.
\end{lemma}
\begin{proof}
The assumption and results do not depend on the
choice of coordinate systems, so
we may assume that $f$ has the form
\eqref{eq:normalcoord}.
Let us assume that $\rank \hess_0 h(0,y)=k$.
Then by the parametrized Morse Lemma
(see \cite[p.502]{ho}, \cite[p.97]{bg}),
there exist a coordinate system 
$\tilde y=(\tilde y_1,\ldots,\tilde y_{m-n+1})$
and a function $\tilde h$ such that
$$
\begin{array}{rcl}
h(x,y)
&=&
\displaystyle
q(\tilde y)+\tilde h(x,\tilde y_{k+1},\ldots,\tilde y_{m-n+1}),\\
q(\tilde y)&=&\sum_{i=1}^ke_i\tilde y_i^2,\quad
\tilde y=(\tilde y_1,\ldots,\tilde y_k),\quad e_i=\pm1
\end{array}
$$
holds.
We rewrite the coordinate
as $(y_1,\ldots,y_{k})=
(\tilde y_1,\ldots,\tilde y_{k})$ and 
$z=(z_1,\ldots,z_{k'})$ $=
(\tilde y_{k+1},\ldots,\tilde y_{m-n+1})$,
where $k'=m-n+1-k$.
Then 
$
f(x,y,z)
=
(x,f_n(x,y,z))
$
has the form
\begin{equation}\label{eq:normal2}
f_n(x,y,z)=q(y)+\tilde{\tilde{h}}(x,z),\quad
q(y)=\sum_{i=1}^ke_iy_i^2,\quad
\hess \tilde{\tilde{h}}(0,z)(0)=0.
\end{equation}
We rewrite $\tilde{\tilde{h}}(x,z)=h(x,z)$.
Furthermore, by Lemma \ref{lem:rankhf},
one can take an adapted 
collection of vector fields
$$
\xi_i=\partial x_i\ (i=1,\ldots,n-1),\ 
\eta_j=\partial y_j\ (j=1,\ldots,k),\ 
\eta_{k+j}=\partial z_j\ (j=1,\ldots,l).
$$
Then we see that
$\spann{\eta_{k+1},\ldots,\eta_{k+l}}=\ker \HH_\eta$
on $S(f)$.
Set
$$
\lambda_j=\det(\xi_1f,\ldots,\xi_{n-1}f,\eta_jf),\quad
j=1,\ldots,m-n+1,\quad
\Lambda=(\lambda_1,\ldots,\lambda_{m-n+1}).
$$
It follows that
$$
\Lambda
=
(2e_1y_1,\ldots,2e_ky_k,
h_{z_1}(x,z),\ldots,h_{z_l}(x,z))
$$
and
$S(f)=\{\Lambda=0\}$.
By non-degeneracy, we have $d\lambda_0\ne0$.
The matrix which represents $d\lambda_0$
is given by
\renewcommand{\arraystretch}{1.5}
\begin{align*}
A&=\left(
\begin{array}{c|c|c}
\phantom{\lefteqn{\dfrac{\dfrac{1}{1}}{\dfrac{1}{1}}}}
\Big((f_n)_{y_ix_j}\Big)_{\substack{i=1,\ldots,k,\\j=1,\ldots,n-1}}&
\Big((f_n)_{y_iy_j}\Big)_{i,j=1,\ldots,k}&
\Big((f_n)_{y_iz_j}\Big)_{\substack{i=1,\ldots,k,\\j=1,\ldots,l}}\\
\hline
\phantom{\lefteqn{\dfrac{\dfrac{1}{1}}{\dfrac{1}{1}}}}
\Big((f_n)_{z_ix_j}\Big)_{\substack{i=1,\ldots,l,\\j=1,\ldots,n-1}}&
\Big((f_n)_{z_iy_j}\Big)_{\substack{i=1,\ldots,l,\\j=1,\ldots,k}}&
\Big((f_n)_{z_iz_j}\Big)_{i,j=1,\ldots,l}
\end{array}
\right)(0)\\
&=\left(
\begin{array}{c|c|c|c}
*&*&\hess q(y)&O\\
\hline
\phantom{\lefteqn{\dfrac{\dfrac{1}{1}}{\dfrac{1}{1}}}}
\Big(h_{z_ix_j}\Big)_{\substack{i=1,\ldots,l,\\j=1,\ldots,l}}&
\Big(h_{z_ix_j}\Big)_{\substack{i=1,\ldots,l,\\j=l+1,\ldots,n-1}}&
O&
\Big(h_{z_iz_j}\Big)_{i,j=1,\ldots,l}
\end{array}
\right)(0)\\
&=:
\left(
\begin{array}{c|c|c|c}
*&*&\hess q(y)&O\\
\hline
M_1&M_2&O&M_3
\end{array}
\right)(0),
\end{align*}
\renewcommand{\arraystretch}{1}
where $O$ stands for a zero matrix.
Since $M_3(0)=O$,
we may assume $M_1$ is regular by a
coordinate change if necessary.
By the implicit function theorem,
there exist functions
$$
x_i(x_{l+1},\ldots,x_{n-1},z)\quad (i=1,\ldots,l),\quad
z=(z_1,\ldots,z_l),
$$
such that
\begin{equation}\label{eq:s2impl}
\begin{array}{l}
h_{z_j}\Big(
x_1(x_{\overrightarrow{l+1}},z),\ldots,
x_l(x_{\overrightarrow{l+1}},z),x_{\overrightarrow{l+1}},z\Big)
=0\quad  (j=1,\ldots,l),\\
\hspace{70mm}
x_{\overrightarrow{l+1}}=(x_{l+1},\ldots,x_{n-1}).
\end{array}
\end{equation}
Then 
$g:=f|_{S(f)}$ is expressed by
$$
g(x_{l+1},z)=f\Big(x_1(x_{\overrightarrow{l+1}}),\ldots,
x_l(x_{\overrightarrow{l+1}}),
x_{\overrightarrow{l+1}},0,z\Big).
$$
Hence the transportation matrix which represent
$d(f|_{S(f)})$ is given by
$$
B=\left(
\begin{array}{ccc|c|ccc}
(x_1)_{x_{l+1}}&\cdots&(x_l)_{x_{l+1}}&
&\\
\vdots&\vdots&\vdots&I&*\\
(x_1)_{x_{n-1}}&\cdots&(x_l)_{x_{n-1}}&
&\\
\hline
(x_1)_{z_{1}}&\cdots&(x_l)_{z_{1}}&&
\displaystyle\sum_{i=1}^l h_{x_i}(x_i)_{z_1}+h_{z_1}
\\
\vdots&\vdots&\vdots&O&\vdots\\
(x_1)_{z_{l}}&\cdots&(x_l)_{z_{l}}&&
\displaystyle\sum_{i=1}^l h_{x_i}(x_i)_{z_l}+h_{z_l}
\end{array}
\right)
=:
\left(
\begin{array}{c|c|c}
*&I&*\\
\hline
N_1&O&v
\end{array}
\right),
$$
where $I$ stands for the identity matrix.
Since 
$\partial z_1,\ldots,\partial z_l$ are contained in $\ker df$
on $S(f)$, the derivatives 
$h_{z_1},\ldots,h_{z_l}$ vanishes on $S(f)$,
and we have
$$
v=N_1\pmt{h_{x_1}\\
\vdots\\
h_{x_l}}.
$$
Hence, by elementary row operations
$B$ changes to
\begin{equation}\label{eq:dfs}
\left(
\begin{array}{c|c|c}
*&I&*\\
\hline
N_1&O&O
\end{array}
\right).
\end{equation}
Thus
$(x,0,z)\in S(f|_{S(f)})$ 
is equivalent to the determinant of
$N_1(x,0,z)$ being zero.
Differentiating \eqref{eq:s2impl},
we have
$$
N_1\trans{M_1}
=
-
\pmt{h_{z_1z_1}&\cdots&h_{z_1z_l}\\
\vdots&\vdots&\vdots\\
h_{z_1z_l}&\cdots&h_{z_lz_l}}.
$$
Since $M_1$ is regular,
$(x,0,z)\in S(f|_{S(f)})$ is equivalent to
$\det \hess h(0,z)=0$.
On the other hand,
$\eta_j\lambda_i = h_{z_iz_j}$
holds on $S(f)$,
and we have
$\hess h(0,z)=\HH_\eta$.
Since 
$\ker dg=\spann{\partial z_{1},\ldots,\partial z_{l}}$ 
by \eqref{eq:dfs},
one can easily see that
the last assertion holds true.
\end{proof}
Set
$$
H=\det \HH_\eta.
$$
\begin{definition}
A non-degenerate singular point $0$ is called
{\em $2$-singular}\/ if
$H(0)=0$.
\end{definition}
This is equivalent to $\ker df_0\cap T_0S(f)\ne\emptyset$.
Set $S_2(f)=\{H=0\}$.
The $2$-singularity of a non-degenerate singular point
does not depend on the choice of $\eta$.
By Lemma \ref{lem:s2}, it follows that
$S_2(f)=S(g)$.
\begin{definition}
A $2$-singular point $0$ is {\em $2$-non-degenerate}\/
if $d(H|_{S(f)})_0\ne0$.
\end{definition}
The condition is equivalent to
$\ker dH_0\not\supset T_0S(f)$.
By the definition, we see that
the $2$-non-degeneracy condition does not depend on the choice
of $\eta$, and if $p$ is $2$-non-degenerate, 
then $S_2(f)$ is a manifold near $p$. 
Moreover, $\rank \HH_\eta(0)=m-n$.
In fact, if we assume that
$\rank \HH_\eta(0)<m-n$, then
all the minor $m-n-1$ determinants of $\HH_\eta(0)$ 
vanish.
Since $dH_0$ is expressed by these minor determinants,
we have $dH_0=0$.

Let $p$ be a $2$-singular point.
Since $H(p)=0$, dimension of $\ker\HH_\eta(p)$ is
positive.
Let $\theta_p$ be a non-zero element of $\ker\HH_\eta(p)$.
\begin{lemma}\label{lem:thetaexist}
If\/ $\rank \HH_\eta=m-n$, then there exists a vector field\/ $\theta$
on\/ $(\R^m,0)$ such that\/
$\theta_p$ generates\/ $\ker \HH_\eta(p)$ when\/
$p\in S_2(f)(=\{H=0\})$.
Namely, 
$\left\langle\theta_p\right\rangle_{\R}=\ker \HH_\eta(p)$.
\end{lemma}
\begin{proof}
The matrix $\HH_\eta$ is symmetric on $S(f)$,
and has only one zero-eigenvalue at $0$.
Thus the eigenvalue $\kappa$, that has minimum absolute value, 
is well-defined on a neighborhood $U$ of $0$, 
and it takes a real value on $U$.
We denote that by $\theta$
the non-zero eigenvector with respect to $\kappa$.
Then $\theta$ is an eigenvector of the zero eigenvalue
on $S_2(f)$ and so, 
one can extend $\theta$ on $(\R^m,0)$, and get
the desired vector field.
\end{proof}
We state a condition that $\theta$ is in the kernel of
$\HH_\eta$.
\begin{lemma}\label{lem:kerh}
For\/ $p\in S(f)$, the condition\/
$
\theta_p\in\ker\HH_\eta(p)
$
is equivalent to\/
$\theta \lambda_1=\cdots=\theta\lambda_{m-n+1}=0$ at\/ $p$.
\end{lemma}
\begin{proof}
Let
$\eta_1,\ldots,\eta_{m-n+1}$ be vector fields
generating $\ker df$, and set
$\theta=\sum_{i=1}^{m-n+1}\theta_i\eta_i$.
Then by \eqref{eq:hh} and symmetry of $\HH_\eta$, 
we see that
$$
\HH_\eta(\theta)
=
\sum_{i=1}^{m-n+1}
\Bigg(\sum_{j=1}^{m-n+1}\theta_j\eta_i\lambda_j\Bigg)
\eta_i
=
\sum_{i=1}^{m-n+1}
\Bigg(\sum_{j=1}^{m-n+1}\theta_j\eta_j\lambda_i\Bigg)
\eta_i
=
\sum_{i=1}^{m-n+1} (\theta\lambda_i)
\eta_i.
$$
Thus the assertion holds.
\end{proof}
If $p$ is a $2$-non-degenerate singular point,
then $S_2(f)$ is a manifold near $p$.
Thus the condition that
$\theta$ is tangent to $S_2(f)$ at a point on $S_2(f)$
is well-defined.
Hence we introduce the definition below.
In what follows,
we denote by $'$ the directional derivative along the direction $\theta$.
Namely, $H'=\theta H$.
Furthermore, $H^{(i)}=(H^{(i-1)})'$ $(i=2,3,\ldots)$ and
$H^{(1)}=H'$, $H^{(0)}=H$.
\begin{definition}
A $2$-non-degenerate singular point $0$ is called $3$-singular
if $\theta(0)\in T_0S_2(f)$.
\end{definition}
Since the $3$-singularity is determined by
$\theta$ at $p$, it does not depend on the extension
of $\theta$, and $S_2(f)$ does not depend on the
extension of $\eta$, so
the $3$-singularity does not depend on 
the extension of $\eta$.
We remark that the $3$-singularity is equivalent to $H'(0)=0$.
Let us set
$S_3(f)=\{q\,|\,\theta_q\in T_qS_2(f)\}$.
Then $S_3(f)$ is determined by $\theta$ on $S_2(f)$.
Thus $S_3(f)$ does not depend on the extension of
$\eta,\theta$.
Furthermore, we see that
$$
S_3(f)
=\{p\in S_2(f)\,|\,H'(p)=0\}
=\{p\in S(f)\,|\,H(p)=H'(p)=0\}.
$$
Using this terminology, $3$-singularity is equivalent to
$0\in S_3(f)$.
Moreover, we have:
\begin{lemma}
It holds that\/
$S_3(f)=S(f|_{S_2(f)})$.
\end{lemma}
\begin{proof}
If $p\in S_2(f)$, it holds that
$\ker d(f|_{S_2(f)})_p=\spann{\theta_p}$.
Thus we obtain the result.
%Since $\dim S_2=n-2$ holds,
%applying Lemma \ref{lem:seigensing} to $f|_{S(f)}$ and $S_2$,
%we see the result.
\end{proof}
\begin{definition}
A $3$-singular point $0$ is {\em $3$-non-degenerate\/} if
$d(H'|_{S_2(f)})_0\ne0$ holds.
\end{definition}
\begin{lemma}\label{lem:3nondeg}
The\/ $3$-non-degeneracy condition on
a\/ $3$-singular point does not 
depend on the extension of\/ $\eta$,
on the extension of\/ $\theta$, nor
on the coordinate system on the target.
\end{lemma}
\begin{proof}
Let $\tilde \theta$ be another extension of $\theta$.
Then 
$\tilde \theta H|_{S_2(f)}=\theta H|_{S_2(f)}$
holds on $S_2(f)$, since 
the $3$-non-degeneracy depends only
on the first differential by $\theta$.
Thus the $3$-non-degeneracy does not depend on
the extension of $\theta$.
On the other hand,
let $\tilde\eta$ be another extension of $\eta$, and
set $\tilde H=\det \HH_{\tilde\eta}$.
Then we have $\tilde H=\alpha H+\beta$, where
$\alpha|_{S(f)}\ne0$ and $\beta|_{S(f)}=0$.
Thus it holds that $\tilde H'=\alpha' H+\alpha H'+\beta'$.
We restrict this formula to $S_2(f)$.
We see that $\beta'=0$ on $S_2(f)$,
because $H=0$ and
$p\in S_2(f)$ then $\theta_p\in T_pS(f)$ holds.
Thus
$$
\tilde H'|_{S_2(f)}=\alpha H'|_{S_2(f)}
$$
holds. 
On the other hand,
if  $0$ is $3$-singular, then by $H'(0)=0$,
we see 
$d(\tilde H'|_{S_2(f)})_0=\alpha d(H'|_{S_2(f)})_0$.
Thus it does not depend on the extension of $\eta$.
\end{proof}
The $3$-non-degeneracy is equivalent to
$\ker d(H')_0\not\supset T_0S_2(f)$. Thus
if $0$ is $3$-non-degenerate, then $S_3(f)$ is
a manifold.
\begin{lemma}\label{lem:ind1}
Let\/ $0$ be a non-degenerate singular point.
Then\/ $0$ is\/ $3$-non-degenerate if and only if\/
$H=H'=0$ at\/ $0$ and\/ $\rank d(H,H')_0|_{T_0S(f)}=2$.
\end{lemma}
\begin{proof}
Since both conditions imply the
$2$-non-degeneracy, we assume $0$ is $2$-non-degenerate.
Since $0$ is non-degenerate, 
we take a coordinate systems on the source
and target such that $f(x,y,z)
=
(x,f_n(x,y,z))
$
has the form \eqref{eq:normal2},
and $(f_n)_{z_1x_1}(0)\ne0$.
Moreover $dH_0\ne0$, we see 
$l=1$.
Then we take an adapted collection of vector fields
\begin{align*}
&\xi_1=\partial x_1,\ 
\xi_i=-(f_n)_{z_1x_i}\partial x_1+(f_n)_{z_1x_1}\partial x_i,\ 
(i=2,\ldots,n-1),\\
&\hspace{40mm}
\eta_j=\partial y_j\ (j=1,\ldots,m-n),\ 
\eta_{m-n+1}=\partial z_1.
\end{align*}
Then we see that
$$
T_0S(f)=\spann{\partial x_2,\ldots,\partial x_{n-1},\partial z_1},
\quad
\partial z_1\in
T_0S_2(f).
$$
We assume that $0$ is $3$-non-degenerate. 
Then $H=H'=0$ at $0$, and $dH'_0|_{T_0S_2(f)}\ne0$ holds.
Thus there exists a vector $\xi\in T_0S_2(f)$ such
that $\xi H'(0)\ne0$.
%Since $S_2(f)=\{H=0\}$, it holds that $\xi H=0$.
%By $dH_0|_{T_0S(f)}\ne0$, it holds
%that 
%$\rank d(H,H')_0|_{T_0S(f)}=2$.
%On the other hand, $\xi H=0$ holds for
%$\xi\in T_0S_2(f)$.
%Hence we see that
%$\rank d(H,H')_0|_{T_0S(f)}=2$ implies
%$d(H')_0|_{S_2(f)}\ne0$.
%%
Since $S_2(f)=\{H=0\}$,\hfill it\hfill holds\hfill 
that\hfill $\xi H=0$.\hfill
By\hfill $dH_0|_{T_0S(f)}\ne0$,\hfill it\hfill holds\hfill
that \\
$\rank d(H,H')_0|_{T_0S(f)}=2$.
On the other hand, $\xi H=0$ 
holds for $\xi\in T_0S_2(f)$. 
Hence we see that 
$\rank d(H,H')_0|_{T_0S(f)}=2$ implies
$d(H')_0|_{S_2(f)}\ne0$.
\end{proof}
We define $(i+1)$-singularity and
$(i+1)$-non-degeneracy inductively.
Let the notion of 
$j$-singularity,
the set of $j$-singular points 
$S_j(f)=\{p\in (\R^m,0)\,|\,H(p)=\cdots=H^{(i-2)}(p)=0\}$ 
as a manifold, and 
$j$-non-degeneracy already be defined for 
$f:(\R^m,0)\to(\R^n,0)$
$(j=1,\ldots,i)$.
Moreover, we assume that these notions 
do not depend on the extensions of $\eta$ and $\theta$.
Here $1$-non-degenerate means non-degenerate,
and $1$-singular point means singular point.
\begin{definition}
An $i$-non-degenerate singular point $0$
is {\em $(i+1)$-singular\/} if
$\theta\in T_0S_i(f)$.
\end{definition}
We remark that, since the $(i+1)$-singularity is defined only by
the condition of
$\theta$ be on $S_i(f)$ and $S_i(f)$ itself, then
it does not depend on the extension
of $\eta$ and $\theta$.
We set $S_{i+1}(f)=\{\theta_p\in T_pS_{i}(f)\}$.
Then $S_{i+1}(f)$ also does not depend on the extension of
$\eta$ and $\theta$, and we have
$$
S_{i+1}(f)=
\{p\in (\R^m,0)\,|\,H(p)=\cdots=H^{(i-1)}(p)=0\}.
$$
\begin{definition}
%An $(i+1)$-singular point $0$ is {\em $(i+1)$-non-degenerate\/}
%if $d(H^{(i-1)}|_{S_{i}(f)})_0\ne0$ holds.
\hfill An\hfill  $(i+1)$-singular\hfill point\hfill 
$0$\hfill is\hfill {\em $(i+1)$-non-degenerate\/}\hfill 
if \\
$d(H^{(i-1)}|_{S_{i}(f)})_0\ne0$ holds.
\end{definition}
\begin{lemma}
The\/ $(i+1)$-non-degeneracy does not depend on the extensions
of\/ $\theta$ and\/ $\eta$.
\end{lemma}
\begin{proof}
We show this for the extension of $\theta$.
Let $\tilde\theta$ be a vector field satisfying that
$\tilde\theta|_{S_2(f)}=\delta\theta|_{S_2(f)}$ $(\delta\ne0)$.
It is enough to show that
$\delta^{i-1}H^{(i-1)}|_{S_{i}(f)}=
\tilde\theta^{i-1}H|_{S_{i}(f)}$.
We show it by induction.
We set
$\tilde\theta=\delta\theta+\gamma$, where
$\gamma$ is a vector field which satisfies $\gamma|_{S_2(f)}=0$.
When $i=2$, we see the conclusion.
We assume that
$(H^{(i-2)}
-
\delta^{i-2}\tilde\theta^{i-2}H)|_{S_{i-1}(f)}=0$.
Then by
$$
 \tilde\theta^{i-1}H|_{S_{i-1}(f)}
=\tilde\theta\,\tilde\theta^{i-2}H|_{S_{i-1}(f)}
=(\delta\theta+\gamma)\tilde\theta^{i-2}H|_{S_{i-1}(f)}
=\delta\theta\,\tilde\theta^{i-2}H|_{S_{i-1}(f)},
$$
we see that
$$
\begin{array}{cl}
&(\delta^{i-1}H^{(i-1)}-\tilde\theta^{i-1}H)|_{S_{i-1}(f)}\\[2mm]
=&
\delta\Big(
\delta^{i-2}\theta H^{(i-2)}-\theta
\tilde\theta^{i-2}H)\Big)\Big|_{S_{i-1}(f)}\\[2mm]
=&
\Big(\delta\theta\big(
\delta^{i-2}H^{(i-2)}-
\tilde\theta^{i-2}H\big)
-(\theta\delta^{i-2})H^{(i-2)}\Big)\Big|_{S_{i-1}(f)}.
\end{array}
$$
By the assumption of induction,
$(\delta^{i-2}H^{(i-2)}
-
\tilde\theta^{i-2}H)|_{S_{i-1}(f)}=0$ holds.
Since $\theta\in TS_{i-1}(f)$ and
$\theta^{i-2}H=0$
hold on
$S_i(f)$, we see that
$$
\Big(\theta\big(
\delta^{i-2}H^{(i-2)}-
\tilde\theta^{i-2}H\big)\Big)\Big|_{S_{i-1}(f)}=
(\theta\delta^{i-2})H^{(i-2)}\Big|_{S_{i-1}(f)}
=0.
$$

(2)
We take another extension $\tilde \eta$ of $\eta$,
and $\det \HH_{\tilde\eta}=\tilde H$.
Then by Lemma \ref{lem:rankhf} we see that
$\tilde H=\alpha H+\beta$ holds, where 
$\alpha|_{S(f)}\ne0$ and $\beta|_{S(f)}=0$.
Then by the same method as in the proof of 
Lemma \ref{lem:3nondeg},
one can see 
$(\alpha H^{(i-1)}-\tilde H^{(i-1)})|_{S_i(f)}=0$,
by using
$S_i(f)=\{p\in S_{i-1}(f)\,|\,H^{(i-2)}(p)=0\}$,
which proves the $(i+1)$-non-degeneracy 
does not depend on
the extension of $\eta$.
\end{proof}
We remark that $(i+1)$-non-degeneracy is equivalent to
$\ker d(H^{(i-1)})_0\not\supset T_0S_{i}$,
and we can continue until $S_i(f)$ becomes a point, namely
$i=n$.
Since $T_0S_n=\{0\}$, the $(n+1)$-singularity
always fails.
On other words,
$n$-non-degeneracy implies $(n+1)$-non-degeneracy and so on.
In fact, by the definition, $n$-non-degeneracy implies
$d(H^{(n-2)}|_{S_{n-1}(f)})_0\ne0$.
Since $S_{n-1}$ is one-dimensional, 
if
$\theta_p\in T_pS_{n-1}(f)$, then
$\spann{\theta_p}= T_pS_{n-1}(f)$ holds, and
$\theta (H^{(n-2)}|_{S_{n-1}(f)})(0)\ne0$
follows.
\begin{lemma}\label{lem:indk}
\toukouchange{
Let us assume that\/ $i\leq n$,
and\/ $0$ is a non-degenerate singular point.
Then the\\
$i$-non-degeneracy\hfill is\hfill equivalent\hfill to\/\hfill
$H(0)\hfill=\hfill H'(0)\hfill=\hfill\cdots\hfill=\hfill
H^{(i-2)}(0)\hfill=\hfill0$\hfill and\\
$\rank d(H,H',\cdots,H^{(i-2)})_0|_{T_0S(f)}=i-1$.}
{Let us assume that\/ $i\leq n$,
and\/ $0$ is a non-degenerate singular point.
Then the\/ $i$-non-degeneracy is equivalent to\/
$H(0)=H'(0)=\cdots=H^{(i-2)}(0)=0$ and\/
$\rank d(H,H',\cdots,H^{(i-2)})_0|_{T_0S(f)}=i-1$.}
\end{lemma}
\begin{proof}
By induction we assume that the conclusion is true for
$1,\ldots,i-1$, and 
that $0$ is an $i$-non-degenerate singular
point.
Then we have
$H(0)=H'(0)=\cdots=H^{(i-2)}(0)=0$ and
$\rank d(H,H',\cdots,H^{(i-3)})_0=i-2$.

We take a coordinate system $(x_1,\ldots,x_{n-2},z)$ of
$S$ satisfying 
$$
\rank 
\pmt
{
H_{x_1}&\cdots&H_{x_{i-2}}\\
H'_{x_1}&\cdots&H'_{x_{i-2}}\\
\vdots&\vdots&\vdots\\
 H^{(i-3)}_{x_1}&\cdots&H^{(i-3)}_{x_{i-2}}}(0)
=i-2$$
and
$\theta=\partial z$.
The transposition of the matrix representation
of $d(H,H',\cdots,H^{(i-2)})_0$ 
with respect to this coordinate system is
$$
\left(
\begin{array}{c|c}
K_1&
\lefteqn{\phantom{\displaystyle\int}}
L_1\\
\hline
K_2&
\lefteqn{\phantom{\displaystyle\int}}
L_2
\end{array}
\right)
:=
\left(
\begin{array}{ccc|cccc}
H_{x_1}&\cdots&H_{x_1}^{(i-3)}&H_{x_1}^{(i-2)}\\
\vdots&\vdots&\vdots&\vdots\\
H_{x_{i-2}}&\cdots&H_{x_{i-2}}^{(i-3)}&H_{x_{i-2}}^{(i-2)}\\
\hline
H_{x_{i-1}}&\cdots&H_{x_{i-1}}^{(i-3)}&H_{x_{i-1}}^{(i-2)}\\
\vdots&\vdots&\vdots&\vdots\\
H_{n-2}&\cdots&H_{x_{n-2}}^{(i-3)}&H_{x_{n-2}}^{(i-2)}\\
H'&\cdots&H^{(i-2)}&H^{(i-1)}
\end{array}
\right).
$$
By elementary matrix operations,
the above matrix is deformed to
$$
\left(
\begin{array}{c|c}
K_1&O\\
\hline
K_2&
\lefteqn{\phantom{\displaystyle\int}}
L_2
-K_2K_1^{-1} L_1\\
\end{array}
\right).
$$
Now applying the implicit function theorem,
we see that
$$
X:=\pmt{
(x_1)_{x_{i+1}}&\cdots&(x_i)_{x_{i+1}}\\
\vdots&\vdots&\vdots\\
(x_1)_{x_{n-2}}&\cdots&(x_i)_{x_{n-2}}\\
(x_1)'&\cdots&(x_i)'}
=
-K_2K_1^{-1}.
$$
Since
$$
%d(H^{(i-2)}|_{S_i})=
\pmt{
(H^{(i-2)}|_{S_i})_{x_{i+1}}\\
\vdots\\
(H^{(i-2)}|_{S_i})_{x_{n-2}}\\
(H^{(i-2)}|_{S_i})'}
=\trans{\Big(XL_1+L_2\Big)}
=\trans{\Big(-K_2K_1^{-1}L_1
+L_2\Big)},
$$
we have the conclusion.

\toukouchange
{For\hfill the\hfill converse,\hfill if\hfill we\hfill 
assume\hfill that\hfill
$H(0)\hfill=\hfill H'(0)\hfill=\hfill\cdots\hfill=\hfill 
H^{(i-2)}(0)\hfill=\hfill 0$\hfill and\\
$\rank d(H,H',\cdots,H^{(i-3)})_0=i-2$,
we can see the conclusion just following the arguments
above from the bottom up.}
{For the converse, if we assume that
$H(0)=H'(0)=\cdots=H^{(i-2)}(0)=0$ and
$\rank d(H,H',\cdots,H^{(i-3)})_0=i-2$,
we can see the conclusion just following the arguments
above from the bottom up.}
\end{proof}
\section{Criteria}
\begin{theorem}
A map-germ\/ $f:(\R^m,0)\to(\R^n,0)$ is\/ $\A$-equivalent to
the\/ $k$-Morin singularity\/ $(k=1,\ldots,n)$ if and only if\/
$0$ is a\/ $k$-non-degenerate singularity but 
not\/ $(k+1)$-singular.
\end{theorem}
\begin{proof}
Since the $k$-non-degenerate conditions and 
$(k+1)$-singularity conditions 
do not depend on the coordinate systems on
the source nor the target space,
the sufficiency is obvious by just checking
the normal form \eqref{eq:morinnor1} of the Morin singularity.
We now show the necessity.
Let us assume that 
$0$ is a $k$-non-degenerate singularity but 
not $(k+1)$-singular.
Since the assumption 
does not depend on the coordinate systems on
the source and the target space, and since $\rank df_0=n-1$,
we take a coordinate system such that \eqref{eq:normalcoord}
holds.
When $k=1$, we see that
$\hess h(0,y)$ at $y=0$ is regular.
By the parametrized Morse lemma, we have the conclusion.
We now assume $2\leq k\leq n$.
It is known that the assumption
does not depend on the extension of $\theta$,
and by Lemma \ref{lem:rankhf} and \eqref{eq:hfkankei},
$H_f$ is multiplied by a non-zero function
when changing $\eta$.
Moreover, the condition also does not depend on the extension of
$\eta$, hence
we may take $\eta_i=\partial y_i$,
and coordinate $z$ such that $\theta=\partial z$ at $0$.
Under this coordinate system, we rewrite $f$ as
$$
f(x,y,z)=(x,h(x,y,z)),\quad
x=(x_1,\ldots,x_{n-1}),\quad
y=(y_1,\ldots,y_{m-n}).
$$
Then $\hess h(0,y,0)$  
is regular at $y=0$, by the parametrized Morse
lemma, we may choose coordinates $y$ such that
$f$ takes the form
$$
\begin{array}{l}
\displaystyle
f(x,y,z)=(x,q(y)+h(x,z)),\ 
q(y)=\sum_{i=1}^{m-n}\pm y_i^2,\\
\hspace{50mm}
x=(x_1,\ldots,x_{n-1}),\ 
y=(y_1,\ldots,y_{m-n}).
\end{array}
$$
Then we see that $\theta=\partial z$.
Under this coordinate system,
$H(0)=\theta^2 h(0)$ and
$\theta H(0)=\theta^3 h(0)$ hold.
Moreover, the $i$-non-degeneracy 
implies that
$$
(\theta h_{x_1},\ldots,\theta h_{x_{n-1}})(0)\ne(0,\ldots,0).
$$
We set
$$
\bar f(x,z)=f(x,0,z)
=(x,g(x,z)):(\R^n,0)\to(\R^n,0).$$
Then $\bar f(x,z)$ at $0$ is an $A_k$-Morin singularity
in the sense of \cite{aksing}.
Because the Jacobian is $\lambda:=\det J_{\bar f}=g_z$,
and $\eta:=\partial z$ generates the kernel of $d\bar f$ at 
the singular set.
Thus by the assumption of $i$-non-degeneracy,
we have
$\lambda=\eta\lambda=\cdots=\eta^{k-1}\lambda=0$,
$\eta^{k}\lambda\ne0$,
$\rank d(\lambda,\eta\lambda,\ldots,\eta^{k-1}\lambda)_0=k$.
Hence, it follows by (\cite[Theorem A1]{aksing}),
that $\bar f(x,z)$ is $\A$-equivalent to
$\big(x,z^{k+1}+\sum_{i=1}^{k-1}x_iz^i\big)$.
Since $\det\hess q(0)\ne0$,
we see the assertion.
%since
%$f(x,y,z)=(x,q(y)+z^{k+1}+\sum_{i=1}^{k-1}x_iz^i)$,
%we can deform to $f$ to normal form by changing 
%only $y$.
\end{proof}
The proof here is based on that of Morin \cite{morin}.
By Lemma \ref{lem:indk}, we have the following:
\begin{theorem}\label{thm:main}
Let\/ $0$ be a non-degenerate singular point of\/
$f:(\R^m,0)\to(\R^n,0)$.
Then\/
$f$ at\/ $0$ is a\/ $k$-Morin singularity\/ $(2\leq k\leq n)$ 
if and only if
both conditions above holds true\/{\rm :}
\begin{enumerate}
\item $H=H'=\cdots=H^{(k-2)}=0$, $H^{(k-1)}\ne0$,
\item $\rank d\big(H,H',\ldots,H^{(k-2)}\big)_0|_{T_0S(f)}=k-1$.
\end{enumerate}
Here, 
$\HH_\eta$ is determined by\/ \eqref{eq:hheta}
for an adapted collection of vector fields\/ $(\xi,\eta)$
with respect to\/ $f$,
$H=\det \HH_\eta$, and\/
$\theta$ is a vector field
such that it generates\/ $\ker\HH_\eta$ on\/ $\{H=0\}$,
and\/ $'$ means the directional derivative along\/ $\theta$.
\end{theorem}
Moreover we have the following corollary.
\begin{corollary}\label{cor:main}
Let\/ $0$ be a singular point of\/
$f:(\R^m,0)\to(\R^n,0)$ satisfying\/ $\rank df_0=n-1$.
Then\/
$f$ is a\/ $k$-Morin singularity\/ $(2\leq k\leq n)$ at\/ $0$
if and only if
\begin{enumerate}
\item[$(a)$]\label{item:ak1}
  $H=H'=\cdots=H^{(k-2)}=0$, $H^{(k-1)}\ne0$,
\item[$(b)$]\label{item:ak2}
  $\rank d(\lambda_1,\ldots,\lambda_{m-n+1},H,H',\ldots,H^{(k-2)})_0
=
m-n+k$,
\end{enumerate}
where\/ $H$ is the same as in Theorem\/ $\ref{thm:main}$.
\end{corollary}
\begin{proof}
Since the condition $(a)$ is the same, we show that
$(b)$ is equivalent to non-degeneracy and the condition
$(2)$ of Theorem \ref{thm:main}.
The condition does not depend on the choice of an adapted 
collection of vector fields,
so we choose 
$\xi_1,\ldots,\xi_{n-1},\eta_1,\ldots,\eta_{m-n},\theta$
satisfying 
$\xi_2,\ldots,\xi_{n-1}\in T_0S(f)$,
since 
$\spann{\eta_1,\ldots,\eta_{m-n}}$ $\cap\, T_0S(f)=\{0\}$.
Then the transportation of the matrix representation of
the differential
$
d(\lambda_1,\ldots,\lambda_{m-n+1},H,H',\ldots,H^{(k-2)})_0
$
has the form
$$
\begin{array}{l}
\left(
\begin{array}{c|c|c|c}
\begin{array}{c}
\xi_1\lambda_1\\
\vdots\\
\xi_1\lambda_{m-n}\\
\end{array}
&
\Big(\xi_j\lambda_i\Big)_{\substack{i=1,\ldots,m-n,\\j=2,\ldots,n-1}}
&
\Big(\eta_j\lambda_i\Big)_{\substack{i=1,\ldots,m-n,\\j=1,\ldots,m-n}}
&
\begin{array}{c}
\lambda_1'\\
\vdots\\
\lambda_{m-n}'
\end{array}
\\
\hline
\xi_1\lambda_{m-n+1}
&
\lefteqn{\phantom{\displaystyle\frac{\frac{1}{1}}{\frac{1}{1}}}}
\Big(\xi_j\lambda_{m-n+1}\Big)_{j=2,\ldots,n-1}
&
\Big(\eta_j\lambda_{m-n+1}\Big)_{j=1,\ldots,m-n}
&
\lambda_{m-n+1}'\\
\hline
\begin{array}{c}
\xi_1 H\\
\vdots\\
\xi_1 H^{(k-3)}
\end{array}
&
\Big(\xi_j H^{(i)}\Big)_{\substack{i=0,\ldots,k-3,\\j=2,\ldots,n-1}}
&
\Big(\eta_j H^{(i)}\Big)_{\substack{i=0,\ldots,k-3,\\j=1,\ldots,m-n}}
&
\begin{array}{c}
H'\\
\vdots\\
H^{(k-2)}
\end{array}\\
\hline
\lefteqn{\phantom{\displaystyle\frac{\frac{1}{1}}{\frac{1}{1}}}}
\xi_1 H^{(k-2)}&
\Big(\xi_j H^{(k-2)}\Big)_{j=2,\ldots,n-1}&
\Big(\eta_j H^{(k-2)}\Big)_{j=1,\ldots,m-n}&
H^{(k-1)}
\end{array}
\right)\\
=
\left(
\begin{array}{c|c|c|c}
\begin{array}{c}
*\\
\vdots\\
*\\
\end{array}
&
O
&
\HH_\eta
&
\begin{array}{c}
0\\
\vdots\\
0
\end{array}
\\
\hline
\xi_1\lambda_{m-n+1}
&
0\hspace{5mm}\cdots\hspace{5mm}0
&
0\hspace{5mm}\cdots\hspace{5mm}0
&
0\\
\hline
\begin{array}{c}
*\\
\vdots\\
*
\end{array}
&
\Big(\xi_j H^{(i)}\Big)_{\substack{i=0,\ldots,k-3,\\j=2,\ldots,n-1}}
&
%\Big(\eta_j H^{(i)}\Big)_{\substack{i=0,\ldots,k-3,\\j=1,\ldots,m-n}}
*&
\begin{array}{c}
0\\
\vdots\\
0
\end{array}\\
\hline
*&
*&
*&
H^{(k-1)}
\end{array}
\right)
\end{array}
$$
at the origin. 
Thus we see that the condition $(b)$ is equivalent to
$\xi_1\lambda_{m-n+1}(0)\ne0$ and
$$
\rank\left(
\begin{array}{ccc}
\xi_2H&\cdots&\xi_{n-1}H\\
\vdots&\vdots&\vdots\\
\xi_2H^{(k-3)}&\cdots&\xi_{n-1}H^{(k-3)}
\end{array}
\right)(0)=k-2.
$$
This is nothing but the non-degeneracy and
the condition 
$(2)$ of Theorem \ref{thm:main}.
\end{proof}

\section{Criteria for small $k$}
In this section, we remark that for small $k$,
the criteria can be simplified.
In what follows, 
for real numbers $a,b\in\R$, the notation $a\sim b$ 
implies $a=0$ is equivalent to $b=0$,
and for functions $f,g$, the notation $f\sim g$ 
implies that $g$ 
is multiplication by a non-zero function $f$.
\subsection{Criterion of the fold}
\begin{corollary}
Let\/ $f=(f_1,\ldots,f_n):(\R^m,0)\to(\R^n,0)$ be a map-germ
satisfying\/
$(df_n)_0=0$ and\/ $\rank df_0=n-1$.
Then\/ $f$ is a fold singularity at\/ $0$ if and only if\/
$\rank \hess_{\hat\eta} f_n=m-n+1$,
where\/
$\hat\eta_1,\ldots,\hat\eta_{m-n+1}$
are vector fields satisfying that\/
$\spann{\hat\eta_1,\ldots,\hat\eta_{m-n+1}}=\ker df$
at\/ $0$.
Here, the number of minus signs in\/ $q$ is equal to 
the number of negative eigenvalues of\/ $\hess_{\hat\eta} f_n$.
\end{corollary}
\begin{proof}
Let
$\xi_1,\ldots,\xi_{n-1},\eta_1,\ldots,\eta_{m-n+1}$ 
be an adapted collection of vector fields
with respect to $f$.
Then $f$ is a fold singularity at $0$ if and only if non-degeneracy holds
with $H(0)\ne0$.
Since $H(0)\ne0$ implies that $\rank \HH_\eta=m-n+1$,
and $d\Lambda_0$ contains $\HH_\eta$,
the non-degeneracy follows from $H(0)\ne0$.
Thus $f$ is the fold if and only if
$H(0)\ne0$.
On the other hand, by $(df_n)_0=0$, we see
$\eta_j\lambda_i(0)=\delta \eta_j\eta_i f_n(0)$,
where
$$
\delta=\det\pmt{
\xi_1f_1&\cdots&\xi_{n-1}f_1\\
\vdots&\vdots&\vdots\\
\xi_1f_{n-1}&\cdots&\xi_{n-1}f_{n-1}}.
$$
Thus 
$H=\det\hess_\eta f_n(0)$.
Furthermore,
since $\eta_1,\ldots,\eta_{m-n+1}$ satisfies
$\eta_i f_n=0$, we see that 
$\eta_i(0)=\hat\eta_i(0)$ implies
$\det\hess_{\hat\eta} f_n(0)=\det\hess_{\eta} f_n(0)$.
\end{proof}
\subsection{Criterion of the cusp}
For a function-germ $t:(\R^n,0)\to(\R,0)$ which has a critical point
at $0$ and a subspace $V\subset T_0\R^n$,
we consider the Hessian matrix $(v_jv_i t)_{(1\leq i,j\leq k)}$
with respect to a basis $v_1,\ldots,v_k$ of $V$,
which is defined by $(\tilde v_j\tilde v_i t)_{(1\leq i,j\leq k)}(0)$,
where $\tilde v_i$ is an extension of $v_i$. 
We remark that since $t$ has a critical point at $0$,
it does not depend on the choice of extensions.
Moreover, the
%since $(v_jv_i t)_{(1\leq i,j\leq k)}$ is multiplied
%by changing basis by $dt_0=0$,
%we see that
$\ker(v_jv_i t)_{(1\leq i,j\leq k)}$
depends only on $V$.
We denote it by $\ker \hess_{V} h(0)$.
\begin{corollary}\label{cor:cusp}
Let\/ $f=(f_1,\ldots,f_n):(\R^m,0)\to(\R^n,0)$ be a map-germ
satisfying\/
$(df_n)_0=0$ and\/ $\rank df_0=n-1$.
Then\/ $f$ is a cusp singularity at\/ $0$, if and only if
for a vector field\/ $\hat\theta$ satisfying 
$$\ker \hess_{\ker df_0} f_n(0)=\langle\hat\theta_0\rangle_{\R},$$
and contained in the\/ $\ker df$ on\/ $S(f)$,
it holds that
\begin{enumerate}
\item\label{item:c2cond3}
%$\theta f_3(0)=\theta^2 f_3(0)=0$,
$\hat\theta^3 f_n(0)\ne0$,
\item\label{item:c2cond4}
$d(\hat\theta f_n)_0\ne 0$.
\end{enumerate}
Here, $\rank \hess_{\ker df_0} f_n(0)=m-n$ and
the number of negative eigenvalues is equal to
the number of minus signs in\/ $q$.
\end{corollary}
\begin{proof}
%(Proof of Corollary \ref{cor:cusp})
The necessity is obvious, we show
the sufficiency.
By Theorem \ref{thm:main}, we show non-degeneracy,
$2$-non-degeneracy and non-$3$-singularity.
Namely,
we show that the conditions (1) and (2) imply non-degeneracy, 
and $H(0)=0$, $H'(0)\ne0$ and $\rank dH_0|_{T_0S(f)}=1$.
If $\theta$ is a generator of the kernel of $\HH_f$,
then since $0$ is $2$-singular i.e., $\theta_0\in T_0S(f)$,
$\rank dH_0|_{T_0S(f)}=1$ follows by $H'(0)\ne0$.
Thus it is enough to show non-degeneracy, $H(0)=0$ and $H'(0)\ne0$.

Before showing that, we give some calculations.
Since the conditions do not depend on the choice of
$\eta$,
we take an adapted collection 
$(\xi_1,\ldots,\xi_{n-1},\eta_1,\ldots,\eta_{m-n},\theta)$
of vector fields with respect to $f$.
Since $\theta$ belongs to the kernel of $\HH_\eta$
on $S_2(f)$, and $0\in S_2(f)$, so it holds that
$$
\theta\lambda_i=0\quad (i=1,\ldots,m-n+1)
\quad \text{at}\quad 0.
$$
On the other hand, $\eta_i f=0$ $(i=1,\ldots,m-n)$
and $\theta f=0$ hold
on $S(f)$, so it holds that
$$
\theta\eta_i f=0\quad(i=1,\ldots,m-n)\quad\text{and}\quad
\theta^2f=0 \quad\text{at}\quad0.
$$
Moreover, by $(df_n)_0=0$ and $[\theta,\eta_i]\in T\R^m$,
it holds that $\eta_i\theta f_n=0$ at $0$.
Thus the bottom column of
$$
\eta_i\lambda_{m-n+1}(0)
=
\det(\xi_1f,\ldots,\xi_{n-1}f,\eta_i\theta f)(0)
$$
is $0$, and we see that
$\eta_i\lambda_{m-n+1}(0)=0$ $(i=1,\ldots,m-n)$.
We remark that the kernel of $\hess_\eta f_n(0)$ is $\theta_0$.

We translate the conditions $H=0$ and $H'\ne0$
using $f_n$.
By the above calculation, it follows that
$$
\begin{array}{rcl}
\theta H(0)
&=&
\theta \det
\left(
\begin{array}{ccc|c}
\eta_1\lambda_1&\cdots &\eta_{m-n}\lambda_1&\theta\lambda_1\\
\vdots&\vdots&\vdots&\vdots\\
\eta_1\lambda_{m-n}&\cdots &\eta_{m-n}\lambda_{m-n}&\theta\lambda_{m-n}\\
\hline
\eta_1\lambda_{m-n+1}&\cdots &\eta_{m-n}\lambda_{m-n+1}&\theta\lambda_{m-n+1}
\end{array}\right)(0)\\
&=&
\det
\left(
\begin{array}{ccc|c}
\eta_1\lambda_1&\cdots &\eta_{m-n}\lambda_1&\theta^2\lambda_1\\
\vdots&\vdots&\vdots&\vdots\\
\eta_1\lambda_{m-n}&\cdots &\eta_{m-n}\lambda_{m-n}&\theta^2\lambda_{m-n}\\
\hline
\eta_1\lambda_{m-n+1}&\cdots &\eta_{m-n}\lambda_{m-n+1}&\theta^2\lambda_{m-n+1}
\end{array}\right)(0)\\
&\sim&
\theta^2\lambda_{m-n+1}(0)\\
&=&
\theta^2\det(\xi_1f,\ldots,\xi_{n-1}f,\theta f)(0)\\
&=&
\det(\xi_1f,\ldots,\xi_{n-1}f,\theta^3 f)(0)\\
&\sim&
\theta^3 f_n(0).
\end{array}
$$
By the same calculation, $H(0)=\theta^2 f_n(0)$ also holds.
Thus it is necessary to show that
$\theta^2 f_n(0)\sim\hat\theta^2 f_n(0)$ and
$\theta^3 f_n(0)\sim\hat\theta^3 f_n(0)$.
We set
$$
\hat\theta=
\sum_{i=1}^{m-n}\alpha_i\eta_i+\beta\theta,
\quad\alpha_1(0)=\cdots=\alpha_{m-n}(0)=0,\ \beta(0)\ne0.
$$
Then we have
$$
\begin{array}{rcl}
\displaystyle
\hat\theta f_n
&=&
\displaystyle
\sum_{i=1}^{m-n}\alpha_i\eta_if_n+\beta\theta f_n,\\
\hat\theta^2 f_n
&=&
\displaystyle
\sum_{i,j=1}^{m-n}
\underline{\alpha_j}
(\eta_j\alpha_i\underline{\eta_if_n}+\underline{\alpha_i}\eta_j\eta_if_n)
+
\sum_{j=1}^{m-n}
\underline{\alpha_j}(\eta_j\beta\underline{\theta f_n}
+\beta\underline{\eta_j\theta f_n})\\
&&\hspace{5mm}
\displaystyle
+\beta\Big(
\sum_{i=1}^{m-n}(\theta\alpha_i\underline{\underline{\eta_if_n}}
+
\underline{\alpha_i}\underline{\theta\eta_i f_n})
+
\theta\beta\underline{\underline{\theta f}}+\beta \underline{\theta^2 f_n}
\Big).
\end{array}
$$
Here, 
we underline the terms that vanish at the origin,
and we put double underlines under the terms that 
vanish at the origin
and whose differentiation along $\theta$ vanishes at the origin.
Thus we have
$$
\hat\theta^3 f_n(0)=\theta\hat\theta^2 f_n(0)
\sim
\theta^3 f_n(0).
$$
By the same reason, we have
$
\hat\theta^2 f_n(0)\sim
\theta^2 f_n(0)
$.

Now we show the non-degeneracy condition.
We have
$$
\begin{array}{rcl}
\displaystyle
d\Lambda_0
&=&
\left(
\begin{array}{c|c|c}
\Big(\xi_j\lambda_i\Big)_{\substack{i=1,\ldots,m-n,\\j=1,\ldots,n-1}}
&
\Big(\eta_j\lambda_i\Big)_{i,j=1,\ldots,m-n}
&
\begin{array}{c}
\lambda_1'\\
\vdots\\
\lambda_{m-n}'
\end{array}\\
\hline
\lefteqn{\phantom{\displaystyle\frac{\frac{1}{1}}{\frac{1}{1}}}}
\Big(\xi_j\lambda_{m-n+1}\Big)_{j=1,\ldots,n-1}
&
\Big(\eta_j\lambda_{m-n+1}\Big)_{j=1,\ldots,m-n}
&
\lambda'_{m-n+1}
\end{array}
\right)(0)\\[10mm]
&=&
\left(
\begin{array}{ccc|ccc|c}
*&\cdots&*
&
&A&&0\\
\hline
\xi_1\lambda_{m-n+1}&\cdots&\xi_{n-1}\lambda_{m-n+1}
&
0&\cdots&0
&
0
\end{array}
\right)(0).
\end{array}
$$
By $2$-non-degeneracy, $A$ is regular, and
the non-degeneracy is equivalent to
$$
(\xi_1\lambda_{m-n+1},\ldots,\xi_{n-1}\lambda_{m-n+1})(0)
\ne(0,\ldots,0).$$
Moreover, by
$$
\xi_i\lambda_{m-n+1}(0)
=
\det(\xi_1f,\ldots,\xi_{n-1}f,\xi_i\theta f)(0)
\sim \xi_i\theta f_n(0),
$$
$\eta_i\theta f_n(0)=0$ $(i=1,\ldots,m-n)$ and 
$\theta^2 f_n(0)=0$,
the condition $(2)$ is equivalent to the non-degeneracy.
\end{proof}
It should be remarked that by Corollary \ref{cor:cusp},
one can easily see that 
$f:(\R^m,0)\to(\R^2,0)$ is a
cusp singularity at $0$ 
if and only if $0$ is non-degenerate and $f|_{S(f)}$
is $\A$-equivalent to $t\mapsto(t^2,t^3)$.
This criteria was also obtained in \cite{nab}.

\section{First degree bifurcation of Lefschetz singularity}
The Lefschetz singularity is a map-germ 
$(\R^4,0)\to(\R^2,0)$
defined by
$$
L
(x_1,x_2,y_1,y_2)
=
(x_1x_2-y_1y_2, x_1y_2+x_2y_1).
$$
This is obtained by considering a map-germ
$\R^4=\C^2\ni(z,w)\mapsto zw\in\C=\R^2$.
From the view point
of low-dimensional topology
there are many studies of bundles on surfaces
with this kind of singular points called the Lefschetz fibrations
 (See \cite{don,gs}, for example.).
The Lefschetz singularity is not a stable germ,
and it is natural to consider stable perturbations
of it. The wrinkling 
$$
L_w
(x_1,x_2,y_1,y_2,s)
=
(x_1x_2-y_1y_2+s(x_1+x_2), x_1y_2+x_2y_1)
$$
due to Lekili \cite{l} is such a move and has been well studied. 
The Lefschetz singularity is not finitely $\A$-determined,
and one cannot obtain a kind of bifurcation diagram.
Let us consider
$$
\begin{array}{l}
\tilde L_n
(x_1,x_2,y_1,y_2,
a_{1},a_{2},
b_{1},b_{2},
c_{2000},\ldots,
d_{2000},\ldots)\\
\displaystyle
\hspace{10mm}=
(x_1x_2-y_1y_2+a_1x_1+a_2x_2,\ x_1y_2+x_2y_1+b_1x_1+b_2x_2)\\
\displaystyle
\hspace{30mm}
+
\Bigg(
\sum_{i+j+k+l=2}^n c_{ijkl}x_1^ix_2^jy_1^ky_2^l,\ 
\sum_{i+j+k+l=2}^n d_{ijkl}x_1^ix_2^jy_1^ky_2^l\Bigg).
\end{array}
$$
Then it holds that
$$
\displaystyle\Bigg(tL(\oplus^4\E_4)+
\sum_{k=1}^4 \R\dfrac{\partial \tilde L_n}{\partial a_k}\Bigg\vert_{p}
+
\R\dfrac{\partial \tilde L_n}{\partial c_{2000}}\Bigg\vert_{p}
+\cdots
+
\R\dfrac{\partial \tilde L_n}{\partial d_{2000}}\Bigg\vert_{p}
+\cdots\Bigg)
+\oplus^2\M_4^{n+1}
=
\oplus^2\E_4,
$$
where $a_3=b_1,a_4=b_2$ and $p=(x_1,x_2,y_1,y_2,0,\ldots,0)$.
Here, 
$\E_4$ is the set of function-germs $(\R^4,0)\to\R$
and $\M_4$ is its unique maximal ideal, and
$tL:\oplus^4\E_4\to\oplus^2\E_4$ stands for the tangential map:
$$
tL
(v_1,v_2,v_3,v_4)
=
\pmt{
a_1+x_2&a_2+x_1&-y_2&-y_1\\
b_1+y_2&b_2+y_1&x_2 &x_1}
\pmt{v_1\\v_2\\v_3\\v_4}.
$$
Thus we would like to say that $\tilde L_n$ is a ``versal-like''
unfolding of $L$ up to $n$-degrees.
See \cite{m} for the definition of the versal unfolding.

In \cite{iikt}, deformations of Brieskorn polynomials
which include the Lefschetz singularity
is considered, and an evaluation of the number of cusp 
appearing on it is obtained.
Here, we would consider the set
$$
\begin{array}{l}
N=\Big\{C=(a_1,a_2,b_1,b_2)\,\Big|\,
\text{there exists }q=(x_1,x_2,y_1,y_2)\in S(\tilde L_1)\text{ such that }\\
\hspace{63mm}
\tilde L_1\text{ at }q\text{ is not the fold nor the cusp.}\Big\}.
\end{array}
$$
We call $N$ the {\em non-cusp locus}.
Although the bifurcation diagram for $L$ cannot be
drawn in any finite dimensional space,
we can draw $N|_{b_2=\ep}\subset\R^3$ for small $\ep$, 
and the author believes that
we might regard $N$ as a $1$ degree bifurcation diagram
of $L$.
\smallskip

We set $\tilde L(x_1,x_2,y_1,y_2)
=
\tilde L_1(x_1,x_2,y_1,y_2,a_{1},a_{2},b_{1},b_{2})$,
regarding $a_{1},a_{2},b_{1},b_{2}$ as constants.
To detect $N$, we consider the following three conditions
for $q=(x_1,x_2,y_1,y_2)$:
\begin{enumerate}
\renewcommand{\labelenumi}{(\roman{enumi})}
\item\label{itm:51} $\rank d\tilde L_q=0$,
\item\label{itm:52} $\rank d\tilde L_q=1$ and $\rank\HH_\eta=0$, 
\item\label{itm:53} $\rank d\tilde L_q=1$, $\rank\HH_\eta(q)=1$ and
$H(q)=\theta H(q)=0$.
\end{enumerate}
See Theorem \ref{thm:main} for the notations.
\smallskip

%%%-- bifurcation 1 --%%%\subsubsection{bif1}
Let $C=(a_1,a_2,b_1,b_2)\in N$ and $q=(x_1,x_2,y_1,y_2)\in S(\tilde L)$
satisfies the condition (iii).
We assume that $a_1+x_2\ne0$, then
$\eta_1=(a_2+x_1)\partial x_1+(a_1+x_2)\partial x_2, 
\eta_2=y_2\partial x_1+(a_1+x_2)\partial y_1,
\eta_3=y_1\partial x_1+(a_1+x_2)\partial y_2$
form a basis of the kernel of $df$ at $p=(x_1,x_2,y_1,y_2)\in S(\tilde L)$.
Moreover, $\partial x_1$ together with $\eta_1,\eta_2,\eta_3$
forms a basis of $T_p\R^4$.
Then we define
$\lambda_i=\det(\tilde L_{x_1},\eta_i\tilde L)$,
$
\HH_{(\eta_1,\eta_2,\eta_3)}=(\eta_i\lambda_j)_{i,j=1,2,3}
$ and
$H=\det\HH_{(\eta_1,\eta_2,\eta_3)}$.
Then we see
\begin{align*}
\lambda_1&=-x_1 y_2+x_2 y_1
           -b_1 x_1+b_2 x_2
           +a_1 y_1-a_2 y_2
           +a_1 b_2-a_2 b_1,\\
\lambda_2&=x_2^2+y_2^2+a_1 x_2+ b_1y_2,\\
\lambda_3&=x_1 x_2+y_1y_2+a_1 x_1+b_1y_1.
\end{align*}
%%%-- bifurcation 1-1 --%%%\subsubsection{bif11}
If $(b_1+2y_2,(a_1+2x_2)y_1)\ne(0,0)$, then we set
$$
\theta
=
\det
\pmt{
\eta_2\lambda_1&\eta_2\lambda_2\\
\eta_3\lambda_1&\eta_3\lambda_2}
\eta_1
+
\det
\pmt{
\eta_3\lambda_1&\eta_3\lambda_2\\
\eta_1\lambda_1&\eta_1\lambda_2}
\eta_2
+
\det
\pmt{
\eta_1\lambda_1&\eta_1\lambda_2\\
\eta_2\lambda_1&\eta_2\lambda_2}
\eta_3.
$$
Then $\theta$ forms a basis of $\ker\HH$ on $S_2(\tilde L)$.
%%%-- bifurcation 1-1-1 --%%%\subsubsection{bif111}
We assume $(x_1,y_1)\ne(0,0)$.
If $y_1=0$, then $x_2=-a_1$ since $\lambda_3=0$.
Thus we may assume $y_1\ne0$.
Moreover, we may assume $b_1+y_2\ne0$, because
if $b_1+y_2=0$ and $x_1\ne0$, then we see $a_1+x_2=0$ 
by $\lambda_3=0$.
If $b_1+y_2=0$ and $x_1=0$, then we see $x_2=0$ 
by $\lambda_2=0$. Then we see $a_1\ne0$.
Then we have $b_2+y_1=0$ since $\lambda_1=0$.
In this case, $H=0$ and $\theta H=0$ can be calculated as
$$
y_1=-(a_2/a_1)y_2,\quad
a_1^2a_2y_2(a_1^2+y_2^2)=0.
$$
Thus we have $y_1=0$ which is a contradiction.
By the above discussion, if $(x_1,y_1)\ne(0,0)$,
then
$\lambda_1=\lambda_2=\lambda_3=0$ can be modified to
\begin{equation}\label{eq:lef10}
x_2=-\dfrac{x_1(a_1x_1+b_1y_1)}{x_1^2+y_1^2},\ 
y_2=-\dfrac{y_1(a_1x_1+b_1y_1)}{x_1^2+y_1^2},\ 
a_2x_1+y_1b_2+x_1^2+y_1^2=0.
\end{equation}
%% 20150316 note %%
Substituting \eqref{eq:lef10} into 
$H=0$, we have
$$
(b_1x_1-a_1y_1)
\big(
a_1y_1(-3x_1^2+y_1^2-2a_2x_1)+
b_1(x_1^3-3x_1y_1^2
+a_2x_1^2-a_2y_1^2)
\big)
=0.$$

If we assume $b_1x_1-a_1y_1=0$, then we obtain 
$a_1+x_2=0$ by \eqref{eq:lef10}.
Thus we may assume $b_1x_1-a_1y_1\ne0$.
Moreover, if we assume
$\big(-3x_1^2+y_1^2-2a_2x_1,x_1^3-3x_1y_1^2+a_2x_1^2-a_2y_1^2\big)
=(0,0)$, then we obtain $x_1=y_1=0$.
Thus we may assume 
$\big(-3x_1^2+y_1^2-2a_2x_1,x_1^3-3x_1y_1^2+a_2x_1^2-a_2y_1^2\big)\ne(0,0)$.
Then we see that $\theta H=0$ on $H=0$ is equivalent to
$(a_2+x_1)(a_2+2x_1)=0$.
Hence we have a part of $N$:
\begin{equation*}\label{eq:lef15}
N_1=\{(a_1,a_2,b_1,b_2)\,|\,
a_1(a_2^2-b_2^2)-2a_2b_1b_2=0\ \text{or}\ 
a_2(a_1^2+b_1^2)-2b_2(a_1+b_1)=0\}.
\end{equation*}
%%%-- bifurcation 1-1-2 --%%%\subsubsection{bif112}
If $x_1=y_1=0$, then $\lambda_3=0$ holds,
and $\lambda_1=\lambda_2=0$ can be modified to
\begin{equation}\label{eq:lef20}
a_2 (b_1+y_2)-b_2 (a_1+x_2)=0,\quad
x_2(a_1+x_2)+y_2 (b_1+y_2)=0.
\end{equation}

%%%-- bifurcation 1-1-2-1 --%%%\subsubsection{bif1121}
If $(a_2,b_2)\ne(0,0)$, since $a_1 + x_2\ne0$, we obtain
$$
x_2=\dfrac{b_2 (a_2 b_1 - a_1 b_2)}{a_2^2 + b_2^2}, 
\qquad
y_2=\dfrac{a_2 (-a_2 b_1 + a_1 b_2)}{a_2^2 + b_2^2}.
$$
Since $a_1 + x_2\ne0$ and $b_1+2y_2\ne0$,
we obtain
a part of $N$:
\begin{equation*}\label{eq:lef25}
N_2=\{(a_1,a_2,b_1,b_2)\,|\,
a_1a_2+b_1b_2=0\}.
\end{equation*}
%If $y_2\ne0$, then \eqref{eq:lef20}
%can be deformed to 
%\begin{equation}\label{eq:lef30}
%b_1=-\dfrac{a_1 x_2+x_2^2+y_2^2}{y_2},\quad
%b_2=-\dfrac{a_2 x_2}{y_2}.
%\end{equation}
%Substituting these equations into $H=0$,
%we obtain $a_1 x_2+x_2^2-y_2^2=0$.
%Combining this equation and \eqref{eq:lef30},
%we obtain $b_1+2y_2=0$.
%Since we assumed $y_1=0$, this is a contradiction.
%
%%%-- bifurcation 1-1-2-2 --%%%\subsubsection{bif1122}
%Any non-cusp locus of the case $(a_2,b_2)=(0,0)$ is contained
%by $N_2$.
%%%-- bifurcation 1-2 --%%%\subsubsection{bif12}
Next we assume $(b_1+2y_2,(a_1+2x_2)y_1)=(0,0)$.
If $(b_1+2y_2,a_1+2x_2)=(0,0)$, then we have $x_2=y_2=0$ by $\lambda_2=0$.
Thus we see $a_1=b_1=0$. 
In this case, $C=(0,0,a_2,b_2)\in N_2$.
We assume $(b_1+2y_2,y_1)=(0,0)$, and $a_1+2x_2\ne0$.
Then we have $x_1=0$ by $\lambda_3=0$, and
$a_1 x_2+x_2^2-y_2^2=0$ by $\lambda_2=0$.
Then $\theta=a_2\eta_2+(a_1+x_2)\eta_3$ is a generator of kernel of 
$\HH_{(\eta_1,\eta_2,\eta_3)}$.
Then
$\theta H$ is a non-zero multiplication of
$a_2(3 a_1^2+7 a_1 x_2+4 x_2^2+2 y_2^2)$.
Substituting $a_1 x_2+x_2^2-y_2^2=0$ into this formula,
we have
$3a_2(a_1+x_2) (a_1+2 x_2)=0$,
which implies $a_2=0$.
Then we have $b_2=0$ by $\lambda_1=0$.
In this case, $C\in N_2$.
On the other hand, we also see that
if $C=(a_1,a_2,b_1,b_2)\in N$ and $q=(x_1,x_2,y_1,y_2)\in S(\tilde L)$
satisfies the condition (ii),
then $C\in N_1\cup N_2$.
Summarizing the above arguments,
if $a_1+x_2\ne0$, then we have a part of the non-cusp locus
$N_1\cup N_2$.
By symmetry, we may interchange the subscript $1$ with $2$.
Thus we obtain another part of $N$
in the case of $a_2+x_1\ne0$:
\begin{equation*}\label{eq:lef40}
N_3=\{(a_1,a_2,b_1,b_2)\,|\,
a_2(a_1^2-b_1^2)-2a_1b_1b_2=0\ \text{or}\ 
a_1(a_2^2+b_2^2)-2b_1(a_2+b_2)=0\}.
\end{equation*}

Next, we assume $a_1+x_2=a_2+x_1=0$
and $(b_1+y_2,b_2+y_1)\ne(0,0)$.
Then by the same method, 
we see $C\in N_1\cup N_2\cup N_3$.
Also if
$a_1+x_2=a_2+x_1=b_1+y_2=b_2+y_1=0$,
then we see $C\in N_1\cup N_2\cup N_3$.
On the other hand, 
if $C\in N$ and $q\in S(\tilde L)$ satisfies 
the condition (i),
then we also see $C\in N_1\cup N_2\cup N_3$.
Summarizing all these arguments,
we have
%\begin{equation}\label{eq:lef50}
\begin{align*}
N=&\lefteqn{N_1\cup N_2\cup N_3}\\
 =&
\{C=(a_1,a_2,b_1,b_2)\,|\,
a_1(a_2^2-b_2^2)-2a_2b_1b_2=0\}\\
 &\hspace{3mm}\cup
\{C\,|\,
a_2(a_1^2+b_1^2)-2b_2(a_1+b_1)=0\}
\cup\{C\,|\,
a_1a_2+b_1b_2=0\}\\
 &\hspace{3mm}\cup
\{C\,|\,
a_2(a_1^2-b_1^2)-2a_1b_1b_2=0\}
\cup
\{C\,|\,
a_1(a_2^2+b_2^2)-2b_1(a_2+b_2)=0\}.
\end{align*}
We draw the pictures of $N|_{b_2=-\ep}$,
$N|_{b_2=0}$ and
$N|_{b_2=\ep}$
in the $(a_1,a_2,b_1)$-space 
for small $\ep$ in Figure \ref{fig:bif}.
Here, the thick line in $N|_{b_2=0}$ stands for
the wrinkling.
\begin{figure}
\centering
\toukouchange{
  \includegraphics[width=.18\linewidth]{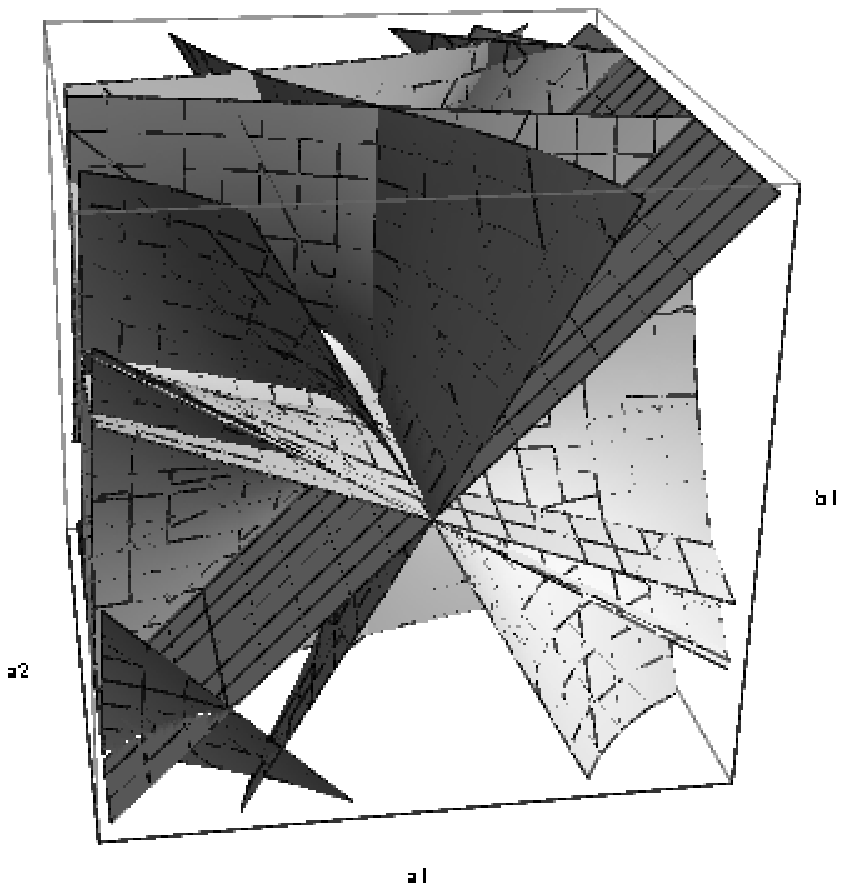}
 \hspace{1mm}
  \includegraphics[width=.18\linewidth]{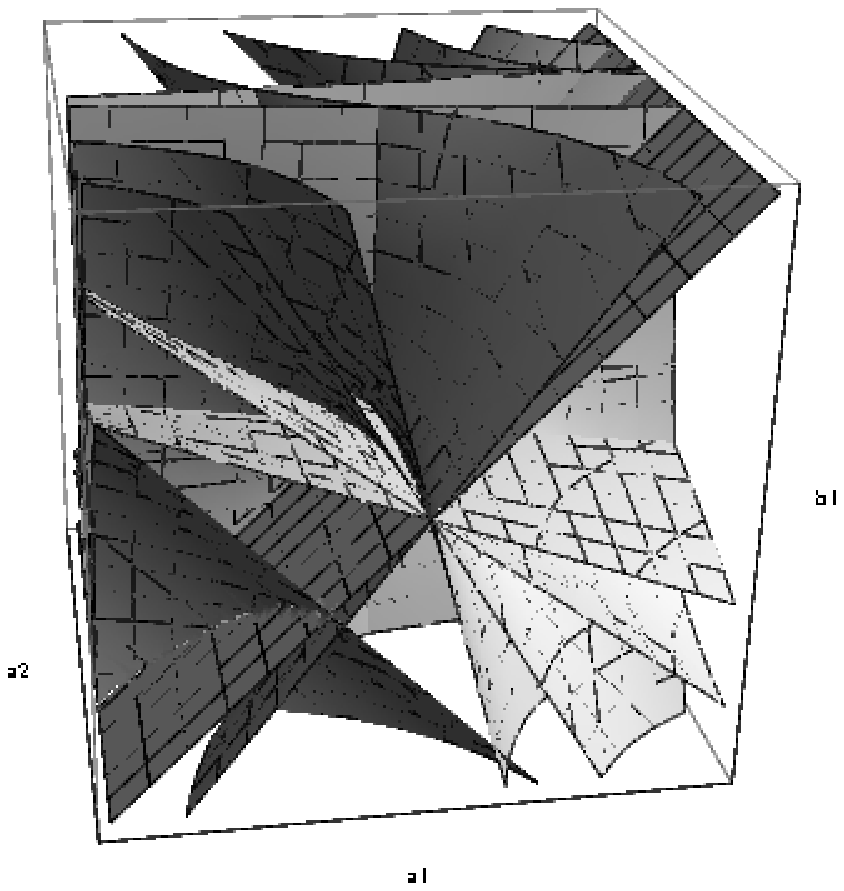}
 \hspace{1mm}
  \includegraphics[width=.18\linewidth]{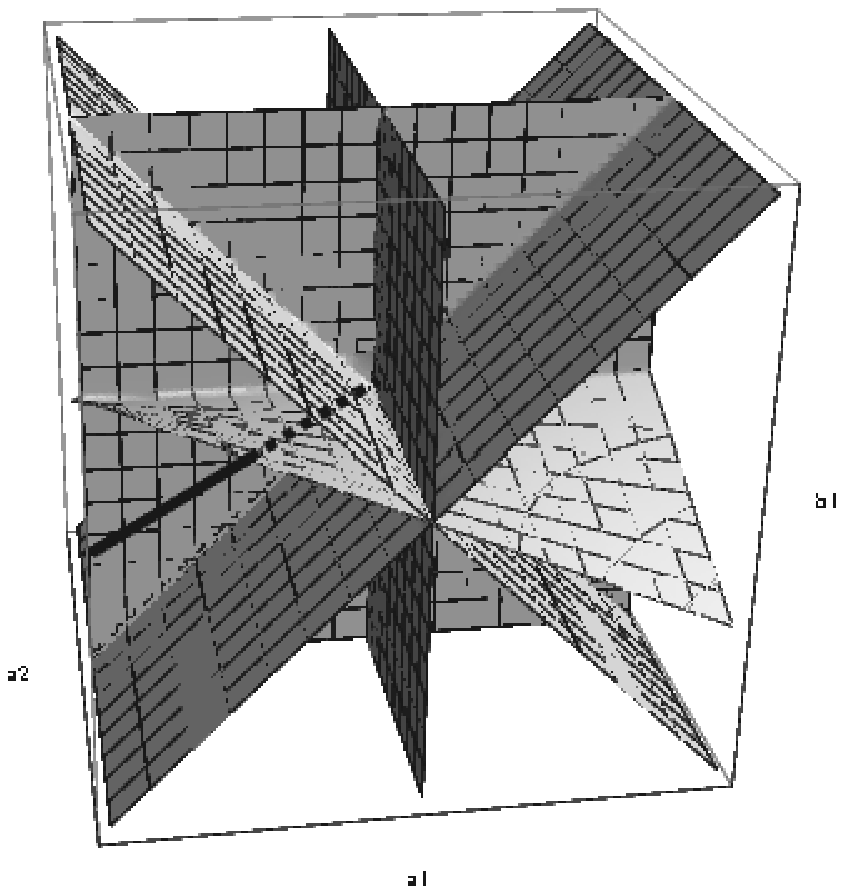}
 \hspace{1mm}
  \includegraphics[width=.18\linewidth]{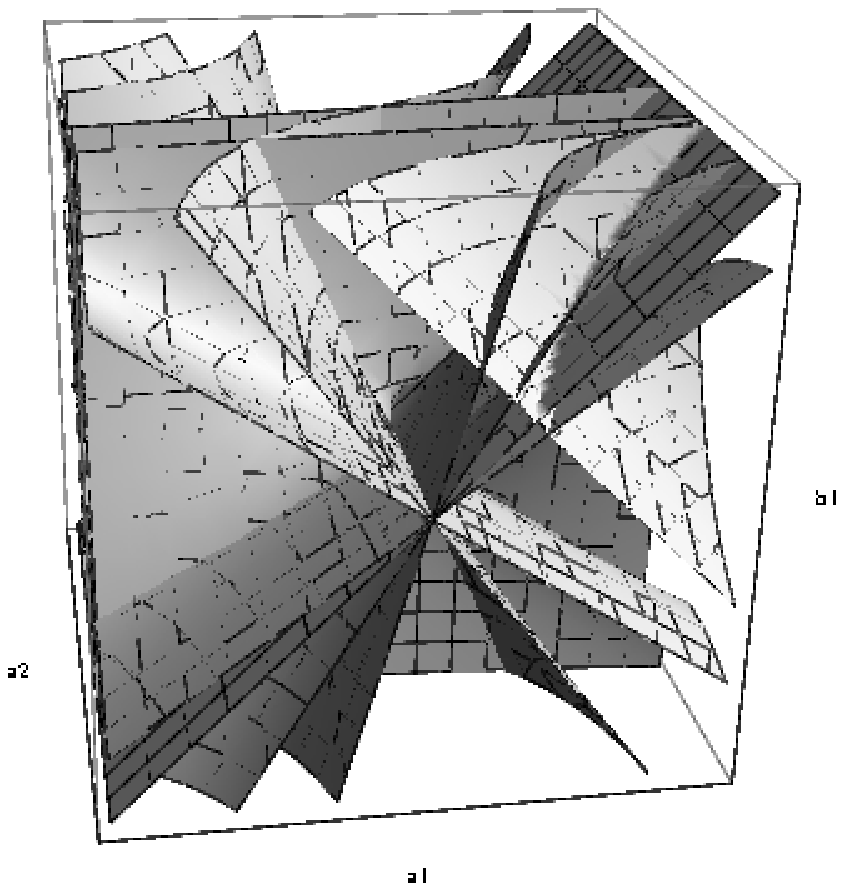}
 \hspace{1mm}
  \includegraphics[width=.18\linewidth]{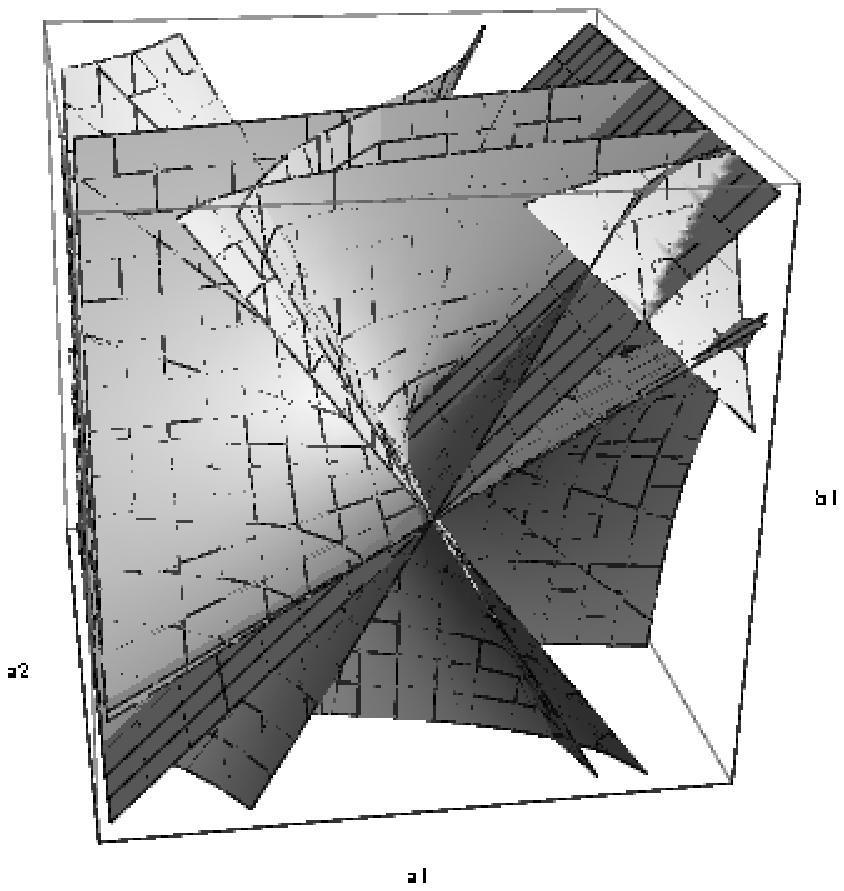}
}
{  \includegraphics[width=.4\linewidth]{figwrin01.eps}
 \hspace{5mm}
  \includegraphics[width=.4\linewidth]{figwrin02.eps}\\
  \includegraphics[width=.4\linewidth]{figwrin03.eps}\\
  \includegraphics[width=.4\linewidth]{figwrin04.eps}
 \hspace{5mm}
  \includegraphics[width=.4\linewidth]{figwrin05.eps}
}
\caption{Non-cusp locus. \toukouchange{}{From upper left to lower right,}
$N|_{b_2=-1/2},\,
N|_{b_2=-1/4},\,
N|_{b_2=0},\,
N|_{b_2=1/4},\,
N|_{b_2=1/2}$}
\label{fig:bif}
\end{figure}

\begin{acknowledgements}
The author thanks
Mar\'{i}a del Carmen Romero Fuster
for fruitful advice and constant encouragement.
He also thanks Kenta Hayano, Toru Ohmoto 
and Kazuto Takao for valuable comments and
encouragements.
\toukouchange{}{
The author is grateful to
the referee for
careful reading and valuable comments.}
\end{acknowledgements}
%\section{thebibliography}

\medskip
\toukoudel{
{\small 
\begin{flushright}
\begin{tabular}{l}
Department of Mathematics,\\
Graduate School of Science, \\
Kobe University, \\
Rokko, Nada, Kobe 657-8501, Japan\\
  E-mail: {\tt sajiO\!\!\!amath.kobe-u.ac.jp}
\end{tabular}
\end{flushright}
}
}
\end{document}